\documentclass[a4paper,11pt]{amsart}

\usepackage{amssymb,amsmath,amsthm,mathrsfs,enumerate,graphicx, color}
\usepackage[pdfpagelabels,colorlinks,linkcolor=blue,citecolor=black,urlcolor=blue]{hyperref}
\usepackage{esint}
\usepackage{tikz}
\usetikzlibrary{arrows}
\usepackage{citeref}

\newtheorem{thm}{Theorem}[section]

\newtheorem{lem}[thm]{Lemma}
\newtheorem{prop}[thm]{Proposition}

\newcommand{\lesi}{\lesssim}

\newcommand{\f}{\frac}

\newcommand{\su}{\subset}
\newcommand{\vc}{\infty}
\newcommand{\Rn}{\mathbb{R}^n}

\title[Higher-order Riesz transforms associated with  Laguerre expansions]{Weighted norm inequalities of higher-order Riesz transforms associated with  Laguerre expansions}         

\author[T. A. Bui]{The Anh Bui}
\address{School of Mathematical and Physical Sciences, Macquarie University, NSW 2109,
	Australia}
\email{the.bui@mq.edu.au}


\begin{document}

\begin{abstract}
Let $\nu=(\nu_1,\ldots,\nu_n)\in (-1,\vc)^n$, $n\ge 1$, and let $\mathcal{L}_\nu$ be a self-adjoint extension of the differential operator 
\[
L_\nu := \sum_{i=1}^n \left[-\frac{\partial^2}{\partial x_i^2} + x_i^2 + \frac{1}{x_i^2}(\nu_i^2 - \frac{1}{4})\right]
\] 
on $C_c^\infty(\mathbb{R}_+^n)$ as the natural domain. The $j$-th partial derivative associated with $L_{\nu}$ is given by
\[
\delta_{\nu_j} = \frac{\partial}{\partial x_j} + x_j-\frac{1}{x_j}\Big(\nu_j + \f{1}{2}\Big), \ \ \ \ j=1,\ldots, n.
\]
In this paper, we investigate the weighted estimates of the higher-order Riesz transforms  $\delta_\nu^k\mathcal L^{-|k|/2}_\nu, k\in \mathbb N^n$, where $\delta_\nu^k=\delta_{\nu_n}^{k_n}\ldots \delta_{\nu_1}^{k_1}$. This completes the description of the boundedness of the higher-order Riesz transforms with the full range $\nu \in (-1,\vc)^n$.
\end{abstract}
\date{}

\maketitle

\tableofcontents

\section{Introduction}\label{sec: intro}
For each $\nu =(\nu_1,\ldots, \nu_n)\in (-1,\vc)^n$, we consider the Laguerre differential operator 
\[
\begin{aligned}
	L_{\nu}
	&=\sum_{i=1}^n\Big[-\f{\partial^2 }{\partial x_i^2} + x_i^2+\frac{1}{x_i^2}\Big(\nu_i^2 - \frac{1}{4}\Big)\Big], \quad \quad x \in (0,\vc)^n.
\end{aligned}
\]
The $j$-th partial derivative associated with $L_{\nu}$ is given by
\[
\delta_{\nu_j} = \frac{\partial}{\partial x_j} + x_j-\frac{1}{x_j}\Big(\nu_j + \f{1}{2}\Big).
\]
Then the  adjoint of $\delta_{\nu_j}$ in $L^2(\mathbb{R}^n_+)$ is
\[
\delta_{\nu_j}^* = -\frac{\partial}{\partial x_j} + x_j-\frac{1}{x_j}\Big(\nu_j + \f{1}{2}\Big).
\]
It is straightforward that
\[
\sum_{i=1}^{n} \delta_{\nu_i}^* \delta_{\nu_i} = L_{\nu} {-2}(|\nu| + n).
\]
Let $k = (k_1, \ldots, k_n) \in \mathbb{N}^n$, $\mathbb N = \{0, 1, \ldots\}$, and $\nu = (\nu_1, \ldots, \nu_n) \in (-1, \infty)^n$ be multi-indices. The Laguerre function $\varphi_k^{\nu}$ on $\mathbb{R}^n_+$ is defined as
\[
\varphi_k^{\nu}(x) = \varphi^{\nu_1}_{k_1}(x_1) \ldots \varphi^{\nu_n}_{k_n}(x_n), \quad x = (x_1, \ldots, x_n) \in \mathbb{R}^n_+,
\]
where $\varphi^{\nu_i}_{k_i}$ are the one-dimensional Laguerre functions
\[
\varphi^{\nu_i}_{k_i}(x_i) = \Big(\f{2\Gamma(k_i+1)}{\Gamma(k_i+\nu_i+1)}\Big)^{1/2}L_{\nu_i}^{k_i}(x_i^2)x_i^{\nu_i + 1/2}e^{-x_i^2/2}, \quad x_i > 0, \quad i = 1, \ldots, n,
\]
given $\nu > -1$ and $k \in \mathbb{N}$, $L_{\nu}^k$ denotes the Laguerre polynomial of degree $k$ and order $\nu$ outlined in  \cite[p.76]{L}.

Then it is well-known that the system $\{\varphi_k^{\nu} : k \in \mathbb{N}^n\}$ is an orthonormal basis of $L^2(\mathbb{R}^n_+, dx)$. Moreover, Each $\varphi_k^{\nu}$ is an eigenfunction of  $L_{\nu}$ corresponding to the eigenvalue of $4|k| + 2|\nu| + 2n$, i.e.,
\begin{equation}\label{eq-eignevalue eigenvector}
	L_{\nu}\varphi_k^{\nu} = (4|k| + 2|\nu| + 2n) \varphi_k^{\nu},
\end{equation}
where $|\nu| = \nu_1 + \ldots + \nu_n$ and  $|k| = k_1 + \ldots + k_n$. The operator $L_{\nu}$ is positive and symmetric in $L^2(\mathbb{R}^n_+, dx)$.

If $\nu = \left(-\frac{1}{2}, \ldots, -\frac{1}{2}\right)$, then $L_{\nu}$ becomes the harmonic oscillator $-\Delta + |x|^2$ on $\mathbb R^n_+$. Note that the Riesz transforms related to the harmonic oscillator $-\Delta + |x|^2$ on $\mathbb R^n$ were investigated in \cite{ST, Th}.

The operator
\[
\mathcal{L}_{\nu}f = \sum_{k\in \mathbb{N}^{n}} (4|k| + 2|\nu| + 2n) \langle f, \varphi_{k}^\nu\rangle \varphi_{k}^\nu 
\]
defined on the domain $\text{Dom}\, \mathcal{L}_{\nu} = \{f \in L^2(\mathbb{R}^n_+): \sum_{k\in \mathbb{N}^{d}} (4|k| + 2|\nu| + 2n) |\langle f, \varphi_{k}^\nu\rangle|^2 < \infty\}$ is a self-adjoint extension of $L_{\nu}$ (the inclusion $C^{\infty}_c(\mathbb{R}^{d}_{+}) \subset \text{Dom}\, \mathcal{L}_{\nu}$ may be easily verified), has the discrete spectrum $\{4\ell + 2|\nu| + 2n: \ell \in \mathbb{N}\}$, and admits the spectral decomposition
\[
\mathcal{L}_{\nu}f = \sum_{\ell=0}^{\infty} (4\ell + 2|\nu| + 2n) P_{\nu,\ell}f,
\]
where the spectral projections are
\[
P_{\nu,\ell}f = \sum_{|k|=\ell} \langle f, \varphi_{k}^\nu\rangle \varphi_{k}^\nu.
\]

Moreover,
\begin{equation}\label{eq- delta and eigenvector}
	\delta_{\nu_j} \varphi_k^\nu =-2\sqrt{k_j} \varphi_{k-e_j}^{\nu+e_j}, \ \ \delta_{\nu_j}^* \varphi_k^\nu =-2\sqrt{k_j+1} \varphi_{k+e_j}^{\nu-e_j},
\end{equation}
where $\{e_1,\ldots, e_n\}$ is the standard basis for $\mathbb R^n$. Here and later on we use the convention that $\varphi_{k-e_j}^{\nu+e_j}=0$ if $k_j-1<0$ and $\varphi_{k+e_j}^{\nu-e_j}=0$ if $\nu_j-1<0$. See for example \cite{NS}.

The primary goal of this paper is to establish the boundedness of the higher-order Riesz transform  
\[
\delta_\nu^k \mathcal{L}_\nu^{-|k|/2}, \quad k\in \mathbb{N}^n,
\]
for the full parameter range \( \nu \in (-1,\infty)^n \), where  
\[
\delta_\nu^k = \delta_{\nu_n}^{k_n} \cdots \delta_{\nu_1}^{k_1}.
\]
The study of Riesz transforms in the setting of orthogonal expansions dates back to the work of Muckenhoupt and Stein~\cite{MS}. Since then, substantial progress has been made, particularly in the context of Riesz transforms associated with Laguerre and Hermite operators. For further developments, see~\cite{Betancor, Betancor2, Muc, NS, NS2, NSj, MST2, MST1, Th} and the references therein.  

In the case \( n=1 \), higher-order Riesz transforms were investigated in~\cite{Betancor} for \( \nu > -1 \). However, the techniques employed there do not work for the  higher dimension case. For \( n \geq 2 \), prior studies have focused only on first-order Riesz transforms of the form  
\[
\delta_{\nu_j} \mathcal{L}_\nu^{-1/2}, \quad j=1,\ldots, n.
\]
Specifically, it was shown in~\cite{NS} that these transforms are Calderón-Zygmund operators under the condition \( \nu_j \in \{-1/2\} \cup (1/2, \infty) \). This result was later extended by the first author to \( \nu_j \in [-1/2, \infty) \) in~\cite{B1}.  

The study of higher-order Riesz transforms in dimensions \( n \geq 2 \) is considerably more challenging due to several factors. Firstly, while the \( L^2 \)-boundedness of first-order Riesz transforms \( \delta_{\nu_j} \mathcal{L}_\nu^{-1/2} \) follows naturally from spectral theory, establishing the \( L^2 \)-boundedness of higher-order Riesz transforms \( \delta_\nu^k \mathcal{L}_\nu^{-|k|/2} \) for \( |k| \geq 2 \) is significantly more involved. Secondly, obtaining kernel estimates for these operators requires precise control over the higher-order derivatives of the heat kernel associated with \( \mathcal{L}_\nu \), which remains an underdeveloped area in the literature.  

In the one-dimensional case \( n=1 \), power-weighted estimates for higher-order Riesz transforms were obtained in~\cite{Betancor} for \( \nu \in (-1, \infty) \). Their approach relied on a comparison with the corresponding higher-order Riesz transform of the Hermite operator in \( \mathbb{R} \). More recently, the first author and X. T. Duong demonstrated that in the higher-dimensional case \( n \geq 1 \), the higher-order Riesz transforms are Calderón-Zygmund operators for \( \nu \in [-1/2, \infty)^n \); see~\cite{BD}.

Motivated by the above research, we aim to establish that the boundedness of the higher-order Riesz transforms \(\delta_\nu^k \mathcal{L}_\nu^{-|k|/2}\) for the full range  \(\nu \in (-1, \infty)^n\).  Our results not only extend those in \cite{NS, B1, Betancor2, BD} but also provide complete the  description of the boundedness properties of higher-order Riesz transforms.

In order to state our main results, we introduce some notation. For  $\nu=(\nu_1,\ldots,\nu_n)\in (-1,\vc)^n$, we set
\begin{equation}\label{eq-gamma nu}
	\gamma_\nu = \max\{\gamma_{\nu_j}: j=1,\ldots,n\}
\end{equation}
where
\begin{equation}\label{eq-gamma nu_j}
	\gamma_{\nu_j} = \begin{cases}
		 -1/2-\nu_j, \ \ & \ \text{if} \  -1<\nu_j<-1/2, \\
		0, \ \ &\ \text{if} \  \nu_j\ge -1/2
	\end{cases}
\end{equation}
for $j=1,\ldots, n$.

Let $k\in \mathbb N^n$. We define $\sigma(k) =(\sigma(k)_1,\ldots, \sigma(k)_n)$, where
\[
\sigma(k)_j=\begin{cases}
	1, \ \ &\text{if $k_j$ is odd};\\
	0, \ \ \ & \text{if $k_j$ is even},
\end{cases}
\]
for $j=1,\ldots,n$.

The following theorem is the main result of this paper. For the definitions of the weight classes $A_p$ and $RH_q$, we refer to Section 2.1.
\begin{thm}\label{mainthm-genral case}
	Let $\nu\in (-1,\vc)^n$ and $k\in \mathbb N^n$.  Then the Riesz transform $\delta_\nu^k \mathcal L_\nu^{-|k|/2}$ is bounded on $L^p_w(\Rn_+)$ for all $\f{1}{1-\gamma_\nu}<p<\f{1}{\gamma_{\nu+\sigma(k)}}$ and $w\in A_{p(1-\gamma_\nu)}\cap RH_{(\f{1}{p\gamma_{\nu+\sigma(k)}})'}$ with the convention $\f{1}{0}=\vc$.
\end{thm}
Before giving some  comments on Theorem \ref{mainthm-genral case}, we would like to introduce the following two maximal functions
\[
\mathcal M_{\mathcal L_\nu} f(x) =\sup_{t>0} |e^{-t\mathcal L_\nu}f(x)|
\]
and
\[
\mathcal M_{k, \mathcal L_\nu} f(x) =\sup_{t>0} |\delta_\nu^ke^{-t\mathcal L_\nu}f(x)|, \ \ \ k\in \mathbb N^n.
\]
By using the kernel estimates in Proposition \ref{prop- delta  pt}, Theorem \ref{thm1- kernel est  n ge 2} and Proposition \ref{prop1-boundedness}, it is straightforward that the maximal function $\mathcal M_{\mathcal L_\nu}$ is bounded on $L^p(\Rn_+)$ for $\f{1}{1-\gamma_\nu}<p<\f{1}{\gamma_\nu}$, while the maximal function $\mathcal M_{k, \mathcal L_\nu}$ is bounded on $L^p(\Rn_+)$ for $\f{1}{1-\gamma_\nu}<p<\f{1}{\gamma_{\nu+\sigma(k)}}$. Note that $\f{1}{\gamma_\nu}\le \f{1}{\gamma_{\nu+\sigma(k)}}$.\\

We turn back to Theorem \ref{mainthm-genral case}.  In the case $n=1$, we have $\sigma(k)=1$ as $k$ is odd and $\sigma(k)=0$ as $k$ is even. Hence, Theorem \ref{mainthm-genral case} reads:
\begin{itemize}
	\item if $k$ is even, the Riesz transform $\delta_\nu^k \mathcal L_\nu^{-|k|/2}$ is bounded on $L^p_w(\Rn_+)$ with $\f{1}{1-\gamma_\nu}<p<\f{1}{\gamma_\nu}$ and $w\in A_{p(1-\gamma_\nu)}\cap RH_{(\f{1}{p\gamma_\nu})'}$;
	\item if $k$ is odd, then $\gamma_{\nu+1}=0$ and hence Riesz transform $\delta_\nu^k \mathcal L_\nu^{-|k|/2}$ is bounded on $L^p_w(\mathbb R_+)$ with $\f{1}{1-\gamma_\nu}<p<1$ and $w\in A_{p(1-\gamma_\nu)}$.
\end{itemize}

\medskip

Our result is still new even in the case \( n=1 \), since Theorem \ref{mainthm-genral case} provides weighted estimates with Muckenhoupt weights instead of power weights as in \cite{Betancor}. When \( k \) is even, the \( L^p \)-boundedness of the Riesz transform \( \delta_\nu^k \mathcal{L}_\nu^{-|k|/2} \) holds for \( \frac{1}{1-\gamma_\nu} < p < \frac{1}{\gamma_\nu} \), which is similar to the boundedness range of the maximal function \( \mathcal{M}_{\mathcal{L}_\nu} \). However, when \( k \) is odd, the Riesz transform \( \delta_\nu^k \mathcal{L}_\nu^{-|k|/2} \) is bounded on \( L^p(\mathbb{R}_+^n) \) for \( \frac{1}{1-\gamma_\nu} < p < \infty \). Surprisingly, in this case, the range of \( p \) is larger than the range \( \frac{1}{1-\gamma_\nu} < p < \frac{1}{\gamma_\nu} \) for the boundedness of the maximal function and the Riesz transforms when $k$ is  even.  However, the boundedness range of the Riesz transform \( \delta_\nu^k \mathcal{L}_\nu^{-|k|/2} \) coincides with that of the maximal function \( \mathcal{M}_{k, \mathcal{L}_\nu} \). One possible explanation for this phenomenon is that we obtain better estimates on the higher-order derivatives of the heat kernel when \( k \) is odd. See Theorem \ref{thm-delta k pnu}.  

For the case \( n \geq 2 \), since \( \gamma_{\nu_j+1} = 0 \) for \( \nu_j \in (-1,\infty) \), we have  
\[
\gamma_{\nu+\sigma(k)} = \max\{\gamma_{\nu_j}: k_j \text{ is even}\}.
\]  
Hence, we conclude the following:
\begin{itemize}
	\item If all entries of \( k \) are even, then \( \sigma(k) = 0 \), and the Riesz transform is bounded on \( L^p_w(\mathbb{R}_+^n) \) for \( \frac{1}{1-\gamma_\nu} < p < \f{1}{\gamma_\nu} \) and \( w \in A_{p(1-\gamma_\nu)} \cap RH_{(\frac{1}{p\gamma_\nu})'} \). In this case, the ranges of \( p \) for the boundedness of the Riesz transform and the maximal function \( \mathcal{M}_{\mathcal{L}_\nu} \) are the same.  
	\item If all entries of \( k \) are odd, then \( \gamma_{\nu+\sigma(k)} = 0 \). In this case, the Riesz transform is bounded on \( L^p_w(\mathbb{R}_+^n) \) for \( \frac{1}{1-\gamma_\nu} < p < \infty \) and \( w \in A_{p(1-\gamma_\nu)} \), and the range of \( p \) is larger than that for the boundedness of the maximal function \( \mathcal{M}_{\mathcal{L}_\nu} \).  
	\item If at least one entry of \( k \) is odd, then we have \( \gamma_{\nu+\sigma(k)} \leq \gamma_\nu \). Theorem \ref{mainthm-genral case} tells us that in this case, the range of \( p \) for the boundedness of the Riesz transform is larger than that for the boundedness of the maximal function \( \mathcal{M}_{\mathcal{L}_\nu} \).  
\end{itemize}  

\bigskip
Some comments on the techniques used in the paper are in order. Firstly, the approach in \cite{Betancor} for the case \( n=1 \) relies on the difference between the kernels of the Riesz transform associated with Laguerre expansions and the Riesz transform associated with Hermite operators. However, this approach might not be applicable to the higher-dimensional case \( n \geq 2 \) or for Muckenhoupt weights.  

Secondly, although the higher-order Riesz transform \( \delta_\nu^k \mathcal{L}_\nu^{-|k|/2} \) in the case \( n \geq 2 \) has been studied in \cite{BD}, the restriction \( \nu \in [-1/2, \infty)^n \) is less challenging since the higher-order Riesz transform is a Calderón-Zygmund operator. Consequently, the boundedness of the Riesz transform does not depend on \( k \).  

To handle the full range \( \nu \in (-1, \infty)^n \), new ideas and techniques are required. In this paper, we employ improved estimates for the derivatives of heat kernels. See, for example, Theorem \ref{thm-delta k pnu} and Proposition \ref{prop-2-stronger estimate on delta k }. Additionally, since the Riesz transform is not a Calderón-Zygmund operator, it is natural to develop and apply techniques from the theory of singular integrals beyond the Calderón-Zygmund setting. For this direction, we refer to \cite{DM, AM1, AM2, BZ} and the references therein.

\bigskip

The paper is organized as follows. In Section 2, we recall the definitions of Muckenhoupt weights, some elementary estimates  and a theorem regarding the boundedness criteria of singular integrals beyond the Calder\'on-Zygmund theory. Section 3 is dedicated to proving kernel estimates related to the heat semigroup $e^{-t\mathcal{L}_\nu}$. Section 4 is devoted to the proofs of Theorem \ref{mainthm-genral case}.

\bigskip
 
 Throughout the paper, we always use $C$ and $c$ to denote positive constants that are independent of the main parameters involved but whose values may differ from line to line. We will write $A\lesi B$ if there is a universal constant $C$ so that $A\leq CB$ and $A\sim B$ if $A\lesi B$ and $B\lesi A$. For $a \in \mathbb{R}$, we denote the integer part of $a$ by $\lfloor a\rfloor$.  For a given ball $B$, unless specified otherwise, we shall use $x_B$ to denote the center and $r_B$ for the radius of the ball.
 
In the whole paper, we will often use the following inequality without any explanation $e^{-x}\le c(\alpha) x^{-\alpha}$ for any $\alpha>0$ and $x>0$.
\section{Preliminaries}

We start with some notations which will be used frequently. For a measurable subset $E\subset \mathbb R^n_+$ and a measurable function $f$ we denote
$$
\fint_E f(x)dx=\f{1}{|E|}\int_E f(x)dx.
$$
Given a ball $B$, we denote $S_j(B)=2^{j}B\backslash 2^{j-1}B$ for $j=1, 2, 3, \ldots$, and we set $S_0(B)=B$.

Let $1\leq q<\infty$. A nonnegative locally integrable function $w$ belongs to the  Muckenhoupt class  $A_q$, say $w\in A_q$, if there exists a positive constant $C$ so that
$$\Big(\fint_B w(x)dx\Big)\Big(\fint_B w^{-1/(q-1)}(x)dx\Big)^{q-1}\leq C, \quad\mbox{if}\; 1<q<\infty$$
and
$$
\fint_B w(x)dx\leq C \mathop{\mbox{ess-inf}}\limits_{x\in B}w(x),\quad{\rm if}\; q=1,
$$
for all balls $B$ in $\mathbb R^d$. We define $A_\infty =\bigcup_{1\le p<\vc} A_p$.


The reverse H\"older classes are defined in the following way: $w\in RH_r, 1 < r < \infty$, if there is a constant $C$ such that for
any ball $B \subset \mathbb R^n_+$,
$$
\Big(\fint_B w^r (x) dx\Big)^{1/r} \leq C \fint_B w(x)dx.
$$
The endpoint $r = \infty$ is given by the condition: $w \in RH_\infty$ whenever, there is a constant $C$ such that for any ball
$B \subset \mathbb R^n_+$,
$$
w(x)\leq C \fint_B w(y)dy  \ \text{for a.e. $x\in B$}.
$$

For $w \in A_\vc$ and $0< p <\infty$, the weighted space $L^p_w(\mathbb R^n_+)$ is defined as  the space of $w(x)dx$-measurable functions $f$ such that
$$\|f\|_{L^p_w(\mathbb R^n_+)}:=\Big(\int_{\mathbb R^n_+} |f(x)|^p w(x)dx\Big)^{1/p}<\infty.$$

It is well-known that the power weight $w(x)=x^\alpha \in A_p$ if and only if $-n<\alpha<n(p-1)$. Moreover, $w(x)=x^\alpha \in RH_q$ if and only if $\alpha q>-n$.

We sum up some of the properties of Muckenhoupt classes and reverse H\"older classes in the following
results. See \cite{Du, JN}.
\begin{lem}\label{weightedlemma1}
	The following properties hold:
	\begin{enumerate}[{\rm (i)}]
		\item $w\in A_p, 1<p<\vc$ if and only if $w^{1-p'}\in A_{p'}$.
		\item $A_1\subset A_p\subset A_q$ for $1\leq p\leq q\leq \infty$.
		\item $ RH_q \su RH_p$ for $1< p\leq q< \infty$.
		\item If $w \in A_p, 1 < p < \vc$, then there exists $1 < q < p$ such that $w \in A_q$.
		\item If $w \in RH_q, 1 < q < \vc$, then there exists $q < p < \infty$ such that $w \in RH_p$.
		\item $A_\vc =\cup_{1\leq p<\vc}A_p = \cup_{1< p< \vc}RH_p$.
		\item Let $1<p_0 < p < q_0<\vc$. Then we have
		$$
		w\in A_{\f{p}{p_0}}\cap RH_{(\f{q_0}{p})'}\Longleftrightarrow
		w^{1-p'}\in A_{\f{p'}{q'_0}}\cap RH_{(\f{p'_0}{p'})'}.
		$$
	\end{enumerate}
\end{lem}

For $r>0$, the Hardy-Littlewood maximal function $\mathcal{M}_r$ is defined by
$$
\mathcal{M}_rf(x)=\sup_{B\ni x}\Big(\f{1}{|B|}\int_B|f(y)|^r\,dy\Big)^{1/r}, \ x\in \mathbb R^n_+,
$$
where the supremum is taken over all balls $B$ containing $x$. When $r=1$, we write $\mathcal{M}$ instead of $\mathcal{M}_1$.

We now record the following results concerning the weak type estimates and the weighted estimates of the maximal functions.
\begin{lem}[\cite{S}]\label{Lem-maximalfunction}
	Let $0< r<\vc$. Then we have for $p>r$,
	$$
	\|\mathcal{M}_rf\|_{L^p(\Rn_+)}\lesi \|f\|_{L^p(\Rn_+)}.
	$$
\end{lem}

We will end this section by some simple estimates whose proofs will be omitted. See for example \cite[Eq. (35)]{B2}.
\begin{lem}\label{lem-elementary lemma}
	For $a\in [0,1/2)$ and $c>0$, define
	\[
	H_{t,a,c}(x,y) = \f{1}{\sqrt t}\exp\Big(-\f{|x-y|^2}{ct}\Big)\Big(1+\f{\sqrt t}{x} \Big)^{a}\Big(1+\f{\sqrt t}{y} \Big)^{a}
	\]
	for $t>0$ and $x,y>0$.
	
	Then we have
	\begin{equation}\label{eq-product H less than H}
		\int_0^\vc  H_{t,a,c}(x,z)H_{t,a,c}(z,y)dz\lesi  H_{t,a,4c}(x,y)
	\end{equation}
	for $t>0$ and $x,y>0$.
\end{lem}

We now recall a theorem which is taken from \cite[Theorem 6.6]{BZ} on a criterion for the singular integrals to be bounded on the weighted Lebesgue spaces.
\begin{thm}\label{BZ-thm}
	Let $1\leq p_0< q_0< \vc$ and let $T$ be a linear operator. Assume that $T$ can be extended to be bounded on $L^{q_0}$. Assume that there exists a family of operators $\{\mathcal{A}_t\}_{t>0}$ satisfying that for $j\geq 2$ and every ball $B$
	\begin{equation}\label{eq1-BZ}
		\Big(\fint_{S_j(B)}|T(I-\mathcal{A}_{r_B})f|^{q_0}dx\Big)^{1/q_0}\leq
		\alpha(j)\Big(\fint_B |f|^{q_0}dx\Big)^{1/q_0},
	\end{equation}
	and
	\begin{equation}\label{eq1-BZ-bis}
		\Big(\fint_{S_j(B)}|\mathcal{A}_{r_B}f|^{q_0}dx\Big)^{1/q_0}\leq
		\alpha(j)\Big(\fint_B |f|^{p_0}dx\Big)^{1/p_0},
	\end{equation}
	for all $f\in C^\vc(\Rn_+)$ supported in $B$. If $\sum_j \alpha(j)2^{jd}<\vc$, then $T$ is bounded on $L^p_w(\mathbb R^n_+)$ for all $p\in (p_0,q_0)$ and $w\in
	A_{\f{p}{p_0}}\cap RH_{(\f{q_0}{p})'}$.
\end{thm}
Note that \cite[Theorem 6.6]{BZ} proves Theorem \ref{BZ-thm} for $q_0=2$, but their arguments also work well for any value $q_0\in (1,\vc)$.

\section{Some kernel estimates and $L^p$-boundedness of integral operators}

\subsection{Some kernel estimates}
This section is devoted to establishing some kernel estimates related to the heat kernel of $\mathcal L_\nu$. These estimates play an essential role in proving our main results. We begin by providing an explicit formula for the heat kernel of $\mathcal L_\nu$.

Let $\nu \in (-1,\vc)^n$. For each $j=1,\ldots, n$, similarly to $\mathcal L_\nu$, denote by $\mathcal L_{\nu_j}$ the self-adjoint extension of the differential operator 
\[
L_{\nu_j} :=  -\frac{\partial^2}{\partial x_j^2} + x_j^2 + \frac{1}{x_j^2}(\nu_j^2 - \frac{1}{4})
\] 
on $C_c^\infty(\mathbb{R}_+)$ as the natural domain. It is easy to see that 
\[
\mathcal L_\nu =\sum_{j=1}^n \mathcal L_{\nu_j}.
\]

Let $p_t^\nu(x,y)$ be the kernel of $e^{-t\mathcal L_\nu}$ and let $p_t^{\nu_j}(x_j,y_j)$ be the kernel of $e^{-t\mathcal L_{\nu_j}}$ for each $j=1,\ldots, n$. Then we have
\begin{equation}\label{eq- prod ptnu}
	p_t^\nu(x,y)=\prod_{j=1}^n p_t^{\nu_j}(x_j,y_j).
\end{equation}
For $\nu_j\ge -1/2$, $j=1,\ldots, n,$, the kernel of $e^{-t\mathcal L_{\nu_j}}$ is given by
\begin{equation}
	\label{eq1-ptxy}
	p_t^{\nu_j}(x_j,y_j)=\f{2(rx_jy_j)^{1/2}}{1-r}\exp\Big(-\f{1}{2}\f{1+r}{1-r}(x_j^2+y_j^2)\Big)I_{\nu_j}\Big(\f{2r^{1/2}}{1-r}x_jy_j\Big),
\end{equation}
where $r=e^{-4t}$ and $I_\alpha$ is the usual Bessel funtions of an imaginary argument defined by
\[
I_\alpha(z)=\sum_{k=0}^\vc \f{\Big(\f{z}{2}\Big)^{\alpha+2k}}{k! \Gamma(\alpha+k+1)}, \ \ \ \ \alpha >-1.
\]
See for example \cite{Dziu, NS}.

Note that for each $j=1,\ldots, n$, we can rewrite the kernel $p_t^{\nu_j}(x_j,y_j)$ as follows
\begin{equation}
	\label{eq2-ptxy}
	\begin{aligned}
		p_t^{\nu_j}(x_j,y_j)=\f{2(rx_jy_j)^{1/2}}{1-r}&\exp\Big(-\f{1}{2}\f{1+r}{1-r}|x_j-y_j|^2\Big)\exp\Big(-\f{1-r^{1/2}}{1+r^{1/2}}x_jy_j\Big)\\
		&\times \exp\Big(-\f{2r^{1/2}}{1-r}x_jy_j\Big)I_{\nu_j}\Big(\f{2r^{1/2}}{1-r}x_jy_j\Big),
	\end{aligned}
\end{equation}
where $r=e^{-4t}$.

The following  properties of the Bessel function $I_\alpha$  with $\alpha>-1$ are well-known and are taken from \cite{L}:
\begin{equation}
	\label{eq1-Inu}
	I_\alpha(z)\sim z^\alpha, \ \ \ 0<z\le 1,
\end{equation}
\begin{equation}
	\label{eq2-Inu}
	I_\alpha(z)= \f{e^z}{\sqrt{2\pi z}}+S_\alpha(z),
\end{equation}
where
\begin{equation}
	\label{eq3-Inu}
	|S_\alpha(z)|\le  Ce^zz^{-3/2}, \ \ z\ge 1,
\end{equation}
\begin{equation}
	\label{eq4-Inu}
	\f{d}{dz}(z^{-\alpha}I_\alpha(z))=z^{-\alpha}I_{\alpha+1}(z)
\end{equation}
and
\begin{equation}
	\label{eq6s-Inu}
	0< I_\alpha(z)-I_{\alpha+2}(z)=\f{2(\alpha+1)}{z}I_{\alpha+1}(z), \ \ \ z>0.
\end{equation}

Let $\nu > -1$, we have
\[
I_\alpha(z)-I_{\alpha+1}(z)=[I_\alpha(z)-I_{\alpha+2}(z)] -[I_{\alpha+1}(z)-I_{\alpha+2}(z)].
\]
Applying \eqref{eq6s-Inu}, 
\[
I_\alpha(z)-I_{\alpha+2}(z) =\f{2(\alpha+1)}{z}I_{\alpha+1}(z).
\]
Hence,
\[
|I_\alpha(z)-I_{\alpha+1}(z)|\le \f{2(\alpha+1)}{z}I_{\alpha+1}(z)+|I_{\alpha+1}(z)-I_{\alpha+2}(z)|.
\]
On the other hand, since $\alpha+1>-1/2$, we have
\[
0< I_{\alpha+1}(z)-I_{\alpha+2}(z)<2(\alpha+2)\f{I_{\alpha+2}(z)}{z}<2(\alpha+2)\f{I_{\alpha+1}(z)}{z}, \ \ \ z>0.
\]
See for example \cite{Na}.

Consequently,
\begin{equation}
	\label{eq5s-Inu}
	|I_\alpha(z)-I_{\alpha+1}(z)|<(4\alpha +6)\f{I_{\alpha+1}(z)}{z}, \ \ \ \alpha>-1, z>0.
\end{equation}

Particularly, in the case $n=1$, from \eqref{eq6s-Inu}, \eqref{eq5s-Inu} and \eqref{eq1-ptxy}, for $\alpha>-1$ we immediately imply the following, for $t>0$ and $x,y \in (0,\vc)$, 
\begin{equation}
	\label{eq6-Inu}
	p_t^\alpha(x,y)-p_t^{\alpha+2}(x,y)= 2(\alpha+1)\f{1-r}{2r^{1/2}xy}p_t^{\alpha+1}(x,y) 
\end{equation}
and
\begin{equation}
	\label{eq5-Inu}
	|p_t^\alpha(x,y)-p_t^{\alpha+1}(x,y)|\lesi \f{1-r}{2r^{1/2}xy}p_t^{\alpha+1}(x,y), 
\end{equation}
where $r=e^{-4t}$.

Due to \eqref{eq- prod ptnu}, once we get estimates for the case $n=1$, the case $n\ge 2$ follows immediately. 	

\medskip

\noindent \underline{\textbf{The case $n=1$.}}

\medskip

In this case, we have
\begin{equation*} 
	\gamma_\nu = \begin{cases}
		 -1/2-\nu, & \ -1<\nu<-1/2, \\
		0, \ \ &\ \nu\ge -1/2.
	\end{cases}
\end{equation*}

We first recall some known estimates in \cite{B1, B2, BD}.
		
		
\begin{prop}
	\label{prop- delta k pt nu -1/2} Let $\nu\ge -1/2$. Then for any $k\in \mathbb N$ we have
		\begin{equation}
		\label{eq-  pt}
		|p_t^\nu(x,y)|\lesi   \f{e^{-t/8}}{t^{(k+1)/2}}\exp\Big(-\f{|x-y|^2}{ct}\Big) \Big(1+ \f{\sqrt t}{x}\Big)^{-(\nu+1/2)}
	\end{equation}
	and
	\begin{equation}
		\label{eq- delta k pt}
		| \delta^k_\nu   p_t^\nu(x,y)|\lesi   \f{1}{t^{(k+1)/2}}\exp\Big(-\f{|x-y|^2}{ct}\Big) \Big(1+ \f{\sqrt t}{x}\Big)^{-(\nu+1/2)}
	\end{equation}
	for $t>0$ and $x,y\in \mathbb R_+$.
\end{prop}
\begin{proof}
	The estimate \eqref{eq-  pt} can be found in \cite[Proposition 3.2]{B1} for $\nu=-1/2$ and in \cite[Lemma 3.3]{B1} for $\nu>-1/2$; while the estimate \eqref{eq- delta k pt} is taken from \cite[Proposition 3.9]{BD}.
\end{proof}
The following estimates are just Lemma 3.6 in \cite{B2}.
\begin{prop}[\cite{B2}]
	\label{prop- delta  pt} Let $\nu>-1$. Then we have
			\begin{equation}
					\label{eq-partial pt}
					\begin{aligned}
							|\partial_t^k p_t^\nu(x,y)|\lesi_{k,\nu} \f{e^{-t/2}}{t^{(k+1)/2}}\exp\Big(-\f{|x-y|^2}{ct}\Big)\Big(1+\f{\sqrt t}{x}\Big)^{\gamma_\nu}\Big(1+\f{\sqrt t}{y}\Big)^{\gamma_\nu}, \ \ \ k\in \mathbb N,
						\end{aligned}
				\end{equation}
	\begin{equation}
		\label{eq- delta  pt}
		| \delta_\nu   p_t^\nu(x,y)|\lesi_\nu   \f{1}{t}\exp\Big(-\f{|x-y|^2}{ct}\Big) \Big(1+ \f{\sqrt t}{y}\Big)^{\gamma_\nu}
	\end{equation}
	and
	\begin{equation}
	\label{eq- delta* pt}
	| \delta_\nu^*p^{\nu+1}_t(x,y)|\lesi_\nu \f{1}{t}\exp\Big(-\f{|x-y|^2}{ct}\Big)\Big(1+\f{\sqrt t}{x} \Big)^{\gamma_\nu}
	\end{equation}
	for all $t>0$ and $x,y\in \mathbb R_+$.
	
\end{prop}

\bigskip

We now take care of the estimate for $\delta p_t^\nu(x,y)$. To do this, from \eqref{eq1-ptxy} we can rewrite $p_t^\nu(x,y)$ as
\begin{equation}\label{eq- form ptxy for derivative}
p_t^\nu(x,y)=\Big(\f{2\sqrt r}{1-r}\Big)^{1/2}\Big(\f{2r^{1/2}}{1-r}xy\Big)^{\nu+1/2}\exp\Big(-\f{1}{2}\f{1+r}{1-r}(x^2+y^2)\Big)\Big(\f{2r^{1/2}}{1-r}xy\Big)^{-\nu}I_\nu\Big(\f{2r^{1/2}}{1-r}xy\Big),
\end{equation}
where $r=e^{-4t}$.

Setting 
$$
H_\nu(r;x,y)=\Big(\f{2r^{1/2}}{1-r}xy\Big)^{-\nu}I_\nu\Big(\f{2r^{1/2}}{1-r}xy\Big),
$$
then
\begin{equation}\label{eq-new formula of heat kernel}
	p_t^\nu(x,y)=\Big(\f{2r^{1/2}}{1-r}\Big)^{1/2}\Big(\f{2r^{1/2}}{1-r}xy\Big)^{\nu+1/2}\exp\Big(-\f{1}{2}\f{1+r}{1-r}(x^2+y^2)\Big)H_\nu(r;x,y). 
\end{equation}
From \eqref{eq-new formula of heat kernel} and \eqref{eq4-Inu}, applying the chain rule,
\begin{equation}\label{eq- chain rule}
	\begin{aligned}
		\partial_x p_t^\nu(x,y)&=  \f{\nu+1/2}{x}p_t^\nu(x,y)- \f{1+r}{1-r}xp_t^\nu(x,y) + \f{2r^{1/2}}{1-r}y p_t^{\nu+1}(x,y),
	\end{aligned}
\end{equation}
which implies
\begin{equation}\label{eq- chain rule 2}
	\begin{aligned}
		\delta_\nu p_t^\nu(x,y)&=  xp_t^\nu(x,y)- \f{1+r}{1-r}xp_t^\nu(x,y) + \f{2r^{1/2}}{1-r}y p_t^{\nu+1}(x,y).
	\end{aligned}
\end{equation}
On the other hand, since $\partial (xf) = f + \partial f$, we have, for $k,\ell \in \mathbb N$ and $\nu>-1$,
\begin{equation}
	\label{eq-formula for delta k xf}
	\delta_\nu^k (xf) = k\delta^{k-1}_\nu f + x\delta_\nu^k f
\end{equation}
and
\begin{equation}\label{eq- del nu and del nu + 1}
	\delta_\nu^k = \delta_{\nu+\ell}^{k} +\f{k\ell}{x}\delta_{\nu+\ell}^{k-1}.
\end{equation}

\begin{prop}
	\label{prop- delta k pt with 1/x k} Let $\nu>-1$. Then for any $k\in \mathbb N\setminus\{0\}$ we have
	\begin{equation}
		\label{eq- delta k pt}
		\f{1}{x^k}| \delta_{\nu+k}   p_t^{\nu+k}(x,y)|\lesi   \f{1}{t^{(k+2)/2}}\exp\Big(-\f{|x-y|^2}{ct}\Big) 
	\end{equation}
	for $t>0$ and $x,y\in \mathbb R_+$.
\end{prop}
\begin{proof}
	 To do this, we consider two cases.

		\noindent\textbf{Case 1: $t \in (0,1)$.} 
		
		From \eqref{eq- chain rule 2} and the facts that $r\sim 1 + r \sim 1 $ and $1-r\sim t$, we have
		\[
		\begin{aligned}
			|\delta_{\nu+k} p_t^{\nu+k}(x,y)|&\lesi   xp_t^{\nu+k}(x,y)+ \f{x}{t}p_t^{\nu+k}(x,y) + \f{y}{t} p_t^{\nu+k+1}(x,y)\\
			&\lesi   \f{x}{t}p_t^{\nu+k}(x,y) + \f{y}{t} p_t^{\nu+k+1}(x,y)\\
			&\lesi   \f{x}{t}p_t^{\nu+k}(x,y) + \f{|y-x|}{t} p_t^{\nu+k+1}(x,y) +\f{x}{t} p_t^{\nu+k+1}(x,y)\\
			&\lesi   \f{x}{t}p_t^{\nu+k}(x,y) + \f{|y-x|}{t} p_t^{\nu+k+1}(x,y),
		\end{aligned}
		\]
	where in the last inequality we used \eqref{eq6s-Inu} and \eqref{eq1-ptxy}.
	
	Therefore, 
		\[
		\f{1}{x^k}|\delta_{\nu+k} p_t^{\nu+k}(x,y)|\lesi \f{1}{tx^{k-1}}p_t^{\nu+k}(x,y) + \f{|y-x|}{t x^k} p_t^{\nu+k+1}(x,y)=: E_1 +E_2.
		\]

		Using \eqref{eq-  pt},
	\begin{equation*}
		\begin{aligned}
			E_2&\lesi \f{1}{x^k}\f{|x-y|}{t}\f{1}{\sqrt t}\exp\Big(-\f{|x-y|^2}{ct}\Big)\Big(1+\f{\sqrt t}{x}\Big)^{-(\nu+k+3/2)}\\
			&\lesi \f{1}{x^k}\f{|x-y|}{t}\f{1}{\sqrt t}\exp\Big(-\f{|x-y|^2}{ct}\Big)\Big(1+\f{\sqrt t}{x}\Big)^{-k}\\
			&\lesi \f{1}{t^{k/2}}\f{|x-y|}{t}\f{1}{\sqrt t}\exp\Big(-\f{|x-y|^2}{ct}\Big)\\
			&\lesi \f{1}{t^{(k+1)/2}} \f{1}{\sqrt t}\exp\Big(-\f{|x-y|^2}{2ct}\Big).
		\end{aligned}
	\end{equation*}
	Similarly,
	\[
	E_1\lesi \f{1}{t^{(k+1)/2}} \f{1}{\sqrt t}\exp\Big(-\f{|x-y|^2}{ct}\Big).
	\]
	
\bigskip	
	
	\noindent\textbf{Case 2: $t\ge 1$.}	In this case, $1+r\sim 1-r\sim 1$ and $r\le 1$. Hence,
	\[
	\begin{aligned}
		|\delta_{\nu+k} p_t^{\nu+k}(x,y)|&\lesi xp_t^{\nu+k}(x,y) +  y p_t^{\nu+k+1}(x,y)\\
		&\lesi xp_t^{\nu+k}(x,y) +  |x-y| p_t^{\nu+k+1}(x,y) +x p_t^{\nu+k+1}(x,y)\\
		&\lesi xp_t^{\nu+k}(x,y) +  |x-y| p_t^{\nu+k+1}(x,y),
	\end{aligned}
	\]
		where in the last inequality we used \eqref{eq6s-Inu} and \eqref{eq1-ptxy}.

It follows that 
	\[
	\f{1}{x^k}|\delta_{\nu+k} p_t^{\nu+k}(x,y)|\lesi \f{1}{x^{k-1}}p_t^{\nu+k}(x,y) +   \f{|x-y|}{x^k} p_t^{\nu+k+1}(x,y).
	\]
	By \eqref{eq-  pt},
	\begin{equation*}
		\begin{aligned}
			\f{|x-y|}{x^k} p_t^{\nu+k+1}(x,y)&\lesi  \f{|x-y|}{x^k}\f{e^{-t/4}}{\sqrt t}\exp\Big(-\f{|x-y|^2}{ct}\Big)\Big(1+\f{\sqrt t}{x}\Big)^{-(\nu+k+3/2)}\\
			&\lesi  \f{|x-y|}{x^k}\f{e^{-t/4}}{\sqrt t}\exp\Big(-\f{|x-y|^2}{ct}\Big)\Big(1+\f{\sqrt t}{x}\Big)^{-k}\\
			&\lesi  \f{|x-y|}{t^{k/2}}\f{e^{-t/4}}{\sqrt t}\exp\Big(-\f{|x-y|^2}{ct}\Big)\\
			&\lesi  \f{1}{t^{(k-1)/2}}\f{e^{-t/4}}{\sqrt t}\exp\Big(-\f{|x-y|^2}{2ct}\Big)\\	
			&\lesi \f{1}{t^{(k+1)/2}} \f{1}{\sqrt t}\exp\Big(-\f{|x-y|^2}{2ct}\Big).
		\end{aligned}
	\end{equation*}
	Similarly,
	\[
	\f{1}{x^{k-1}}p_t^{\nu+k}(x,y)\lesi \f{1}{t^{(k+1)/2}} \f{1}{\sqrt t}\exp\Big(-\f{|x-y|^2}{ct}\Big).
	\]

	\bigskip

	This completes our proof.
\end{proof}

\begin{prop}\label{prop-1}
	Let $ \nu>-1$. Then for $k\in \mathbb N \backslash\{0\}$ we have
	\[
	\f{1}{x}|\delta_{\nu+1}^k p_t^{\nu+1}(x,y)|
	\lesi  \f{1}{t^{(k+2)/2}}  \exp\Big(-\f{|x-y|^2}{ct}\Big) \ \ \ \text{if $k$ is odd}
	\]
	and
	\[
	\f{1}{x}|\delta_{\nu+1}^k p_t^{\nu+1}(x,y)|
	\lesi  \f{1}{t^{(k+2)/2}} \exp\Big(-\f{|x-y|^2}{ct}\Big)\Big(1+\f{\sqrt t}{x}\Big)^{-\nu-1/2} \ \ \ \text{if $k$ is even}.
	\]
\end{prop}
\begin{proof}
	We write
	\[
	\delta_{\nu+1}^{k} e^{-t\mathcal L_{\nu+1}}= \delta_{\nu+1}^{k} e^{-\f{t}{2}\mathcal L_{\nu+1}}e^{-\f{t}{2}\mathcal L_{\nu+1}}=\delta_{\nu+1}^{k-1} e^{-\f{t}{2}(\mathcal L_{\nu+2}+2)} \delta_{\nu+1} e^{-\f{t}{2}\mathcal L_{\nu+1}}.
	\]
	Since $\delta_{\nu+1} p^{\nu+1}_{t/2}(x,y)$ satisfies the Gaussian upper bound due to Proposition \ref{prop- delta k pt nu -1/2}, we need only to show that 
	\begin{equation}\label{eq1- alternate estimate}
		\f{1}{x}\delta^{k-1}_{\nu+1} p^{\nu+2}_{t/2}(x,y)\lesi \f{1}{t^{(k+2)/2}}  \exp\Big(-\f{|x-y|^2}{ct}\Big) \ \ \ \text{if $k$ is odd}
	\end{equation}
	and
	\begin{equation}\label{eq2- alternate estimate}
		\f{1}{x}\delta^{k-1}_{\nu+1} p^{\nu+2}_{t/2}(x,y)\lesi  \f{1}{t^{(k+2)/2}} \exp\Big(-\f{|x-y|^2}{ct}\Big)\Big(1+\f{\sqrt t}{x}\Big)^{-\nu-1/2} \ \ \ \text{if $k$ is even}.
	\end{equation}
	
	Indeed, applying \eqref{eq- del nu and del nu + 1}, we have
	\[
	\f{1}{x}\delta^{k-1}_{\nu+1} p^{\nu+2}_{t/2}(x,y)=\f{1}{x}\delta^{k-1}_{\nu+2} p^{\nu+2}_{t/2}(x,y)+\f{k-1}{x^2} \delta^{k-2}_{\nu+2} p^{\nu+2}_{t/2}(x,y)
	\]
	Using the upper bound in Proposition \ref{prop- delta k pt nu -1/2} we have
	\[
	\begin{aligned}
		\f{1}{x}|\delta^{k-1}_{\nu+2} p^{\nu+2}_{t/2}(x,y)|&\lesi \f{1}{x}\f{1}{t^{k/2}}\exp\Big(-\f{|x-y|^2}{ct}\Big)\Big(1+\f{\sqrt t}{x}\Big)^{-(\nu+5/2)}\\
		&\lesi \f{1}{x}\f{1}{t^{k/2}}\exp\Big(-\f{|x-y|^2}{ct}\Big)\Big(1+\f{\sqrt t}{x}\Big)^{-1}\\
		&\lesi  \f{1}{t^{(k+1)/2}}\exp\Big(-\f{|x-y|^2}{ct}\Big).
	\end{aligned}
	\]
	For the second term, we consider two cases.
	
	\textbf{Case 1: $k$ is even}
	
	Using the upper bound of  $\delta^{k-2}_{\nu+2} p^{\nu+2}_{t/2}(x,y)$ in Proposition \ref{prop- delta k pt nu -1/2}, 
	\[
	\begin{aligned}
		\f{1}{x^2} |\delta^{k-2}_{\nu+2} p^{\nu+2}_{t/2}(x,y)|&\lesi \f{1}{x^2}\f{1}{t^{(k-1)/2}}\exp\Big(-\f{|x-y|^2}{ct}\Big)\Big(1+\f{\sqrt t}{x}\Big)^{-(\nu+5/2)}\\
		&\lesi  \f{1}{t^{(k+1)/2}}\exp\Big(-\f{|x-y|^2}{ct}\Big)\Big(1+\f{\sqrt t}{x}\Big)^{-(\nu+1/2)}.
	\end{aligned}
	\]
	This ensures \eqref{eq2- alternate estimate}.
	
	\bigskip
	
	\textbf{Case 2: $k$ is odd}
	
	In this cases, we continue to write
	\[
	\f{1}{x^2} \delta^{k-2}_{\nu+2} e^{-\f{t}{2}\mathcal L_{\nu+2}}=\f{1}{x^2}\delta^{k-3}_{\nu+2} e^{-\f{t}{4}(\mathcal L_{\nu+3}+2)}\circ \delta_{\nu+2} e^{-\f{t}{4}\mathcal L_{\nu+2}}.
	\]
	Since the kernel of $\delta_{\nu+2} e^{-\f{t}{4}\mathcal L_{\nu+2}}$ satisfies a Gaussian upper bound (Proposition \ref{prop- delta k pt nu -1/2}), it suffices to prove that the kernel $\displaystyle \f{1}{x^2}\delta^{k-3}_{\nu+2}p_{t/4}^{\nu+3}(x,y)$ also satisfies a Gaussian upper bound. To do this, using \eqref{eq- del nu and del nu + 1} to obtain
	\[
	\delta^{k-3}_{\nu+2} =\Big(\delta_{\nu+3}+\f{1}{x}\Big)^{k-3}= \delta^{k-3}_{\nu+3}+\f{k-3}{x}\delta^{k-4}_{\nu+3}, 
	\]
	we have
	\[
	\f{1}{x^2}\delta^{k-3}_{\nu+2}p_{t/4}^{\nu+3}(x,y)=\f{1}{x^2}\delta^{k-3}_{\nu+3}p_{t/4}^{\nu+3}(x,y) +\f{k-3}{x^3}\delta^{k-4}_{\nu+3}p_{t/4}^{\nu+3}(x,y).
	\]
	For the first term, using Theorem \ref{prop- delta k pt},
	\[
	\begin{aligned}
		\f{1}{x^2}|\delta^{k-3}_{\nu+3}p_{t/4}^{\nu+3}(x,y)|&\lesi \f{1}{x^2}\f{1}{t^{(k-2)/2}}\exp\Big(-\f{|x-y|^2}{ct}\Big)\Big(1+\f{\sqrt t}{x}\Big)^{-(\nu+7/2)}\\
		&\lesi \f{1}{x^2}\f{1}{t^{(k-2)/2}}\exp\Big(-\f{|x-y|^2}{ct}\Big)\Big(1+\f{\sqrt t}{x}\Big)^{-2}\\
		&\lesi  \f{1}{t^{k/2}}\exp\Big(-\f{|x-y|^2}{ct}\Big).
	\end{aligned}
	\]
	It remains to prove the second term $\displaystyle \f{k-3}{x^3}\delta^{k-4}_{\nu+3}p_{t/4}^{\nu+3}(x,y)$ has a Gaussian upper bound. By iteration and the fact that $k$ is odd, we will reduce to prove that 
	\[
	\f{1}{x^{\ell+1}}|\delta_{\nu+\ell+1}p_{t/2^\ell}^{\nu+\ell+1}(x,y)|\lesi \f{1}{t^{(\ell+3)/2}}\exp\Big(-\f{|x-y|^2}{ct}\Big)
	\]
	where $k=2\ell+1$. This is true due to Proposition \ref{prop- delta k pt with 1/x k}.

	This completes our proof.
\end{proof}

\begin{thm}\label{thm-delta k pnu}
		Let $\nu>-1$ and $k\ge 1$. Then we have
	\[
	|\delta_{\nu}^k p_t^{\nu}(x,y)|
	\lesi  \f{1}{t^{(k+1)/2}}  \exp\Big(-\f{|x-y|^2}{ct}\Big)\Big(1+\f{\sqrt t}{y}\Big)^{\gamma_\nu} \ \ \ \text{if $k$ is odd}
	\]
	and
	\[
	|\delta_{\nu}^k p_t^{\nu}(x,y)|
	\lesi  \f{1}{t^{(k+1)/2}}  \exp\Big(-\f{|x-y|^2}{ct}\Big)\Big(1+\f{\sqrt t}{x}\Big)^{\gamma_\nu}\Big(1+\f{\sqrt t}{y}\Big)^{\gamma_\nu} \ \ \ \text{if $k$ is even}.
	\]
\end{thm}
\begin{proof}
	
	If $\nu\ge -1/2$, the theorem follows directly from \eqref{eq- delta k pt}. It remains to prove the theorem for $-1<\nu<-1/2$. We will do this by induction.
	
	$\bullet$ The estimate is true for $k=1$ due to \eqref{eq- delta  pt}.
	
	$\bullet$ Assume that the estimate is true  for all $k=1,\ldots, \ell$ for some $\ell\ge 1$. We now prove the estimate for $k=\ell+1$. In this case, we have
	\[
	\delta_{\nu}^{\ell+1} e^{-t\mathcal L_\nu} = \delta_{\nu}^{\ell} e^{-\f{t}{2}(\mathcal L_{\nu+1}+2)} \circ \delta_{\nu}e^{-\f{t}{2}\mathcal L_{\nu}}.   
	\]
	Due to \eqref{eq- delta  pt}, it suffices to prove that 
	\[
	|\delta_\nu^\ell p_{t/2}^{\nu+1}(x,y)|\lesi \f{1}{t^{(\ell+1)/2}}\exp\Big(-\f{|x-y|^2}{ct}\Big) \ \ \text{if $\ell+1$ is odd};
	\]
	and
	\[
	|\delta_\nu^\ell p_{t/2}^{\nu+1}(x,y)|\lesi \f{1}{t^{(\ell+1)/2}}\exp\Big(-\f{|x-y|^2}{ct}\Big) \Big(1+\f{\sqrt t}{x}\Big)^{-\nu-1/2}\ \ \text{if $\ell+1$ is even};
	\]
	To do this, applying \eqref{eq- del nu and del nu + 1} we have
	\[
	\delta_{\nu}^{\ell} e^{-\f{t}{2}\mathcal L_{\nu+1}} = \delta_{\nu+1}^{\ell} e^{-\f{t}{2}\mathcal L_{\nu+1}} +\f{\ell}{x} \delta_{\nu}^{\ell-1} e^{-\f{t}{2}\mathcal L_{\nu+1}}.
	\]
	Then the two estimates above are just direct consequences of  Proposition \ref{prop- delta k pt nu -1/2} and Proposition \ref{prop-1}.
	
	This completes our proof. 
\end{proof}

The following estimate can be viewed as an improved version of Proposition \ref{prop- delta k pt nu -1/2} if the number of derivative is odd.
\begin{prop}\label{prop-2-stronger estimate on delta k }
	Let $\nu>-1/2$. For \textbf{ any odd} $k\in \mathbb N$, we have
	\begin{equation}\label{eq- stronger estimate delta k p}
 |\delta_{\nu}^k p_t^{\nu}(x,y)|
	\lesi  \f{1}{t^{(k+1)/2}} \exp\Big(-\f{|x-y|^2}{ct}\Big) \Big(1+\f{\sqrt t}{x}\Big)^{-(\nu+3/2)}
	\end{equation}
	for all $x,y\in \mathbb R_+$ and $t>0$.
	
	Consequently, for any \textbf{ any odd} $k\in \mathbb N$,
	\[
	\f{1}{x}|\delta_{\nu}^k p_t^{\nu}(x,y)|
	\lesi  \f{1}{t^{(k+2)/2}} \exp\Big(-\f{|x-y|^2}{ct}\Big) \Big(1+\f{\sqrt t}{x}\Big)^{-(\nu+1/2)}
	\]
	for all $x,y\in \mathbb R_+$ and $t>0$.
	
\end{prop}
\begin{proof}
If $x\ge \sqrt t$ then 
$$
1+\f{\sqrt t}{x} \sim 1.
$$	
Hence, \eqref{eq- stronger estimate delta k p} is just a consequence of Proposition \ref{prop- delta k pt nu -1/2}. For this reason, we need only to prove the inequality when $x<\sqrt t$. We will do this by induction.

\bigskip

\noindent $\bullet$ We first prove \eqref{eq- stronger estimate delta k p} for $k=1$. We  consider two cases: $t\in (0,1)$ and $t\ge 1$.
	
	\medskip
	
\noindent \textbf{Case 1: $t\in (0,1)$}

\medskip

From \eqref{eq- chain rule 2} and the facts $r\sim 1$ and $1-r\sim t$ for $t\in (0,1)$, 
\[
\begin{aligned}
	|\delta_\nu p_t^\nu(x,y)|&\lesi xp_t^\nu(x,y) +\f{x}{t}p_t^\nu(x,y) +\f{y}{t}p_t^{\nu+1}(x,y)\\
	&\lesi  \f{x}{t}p_t^\nu(x,y)+\f{x}{t}p_t^{\nu+1}(x,y) +\f{|y-x|}{t}p_t^{\nu+1}(x,y).
\end{aligned}
\]
Applying \eqref{eq-  pt}, we further imply, for $x<\sqrt t$,		
\[
\begin{aligned}
	\f{x}{t}p_t^\nu(x,y)+\f{x}{t}p_t^{\nu+1}(x,y) &\lesi \f{x}{t}\f{1}{\sqrt t} \exp\Big(-\f{|x-y|^2}{ct}\Big) \Big(1+\f{\sqrt t}{x}\Big)^{-(\nu+1/2)}\\
	&\lesi  \f{1}{t} \exp\Big(-\f{|x-y|^2}{ct}\Big) \Big(1+\f{\sqrt t}{x}\Big)^{-(\nu+3/2)}
\end{aligned}
\]
and
\[
\begin{aligned}
	\f{|y-x|}{t}p_t^{\nu+1}(x,y) 
	&\lesi  \f{|y-x|}{\sqrt t}\f{1}{t} \exp\Big(-\f{|x-y|^2}{ct}\Big) \Big(1+\f{\sqrt t}{x}\Big)^{-(\nu+3/2)}\\
	&\lesi   \f{1}{t} \exp\Big(-\f{|x-y|^2}{2ct}\Big) \Big(1+\f{\sqrt t}{x}\Big)^{-(\nu+3/2)}.
\end{aligned}
\]			
This ensures \eqref{eq- stronger estimate delta k p} for $k=1$ in this case.

\noindent \textbf{Case 2: $t>1$}

\medskip

From \eqref{eq- chain rule 2} and the facts $r\sim 1$ and $1-r\sim 1+r \sim 1$ for $t>1$, 
\[
\begin{aligned}
	|\delta_\nu p_t^\nu(x,y)|&\lesi xp_t^\nu(x,y) + yp_t^{\nu+1}(x,y)\\
	&\lesi  xp_t^\nu(x,y)+xp_t^{\nu+1}(x,y) +|y-x|p_t^{\nu+1}(x,y).
\end{aligned}
\]
Applying \eqref{eq-  pt}, for $x<\sqrt t$ we further imply		
\[
\begin{aligned}
	xp_t^\nu(x,y)+xp_t^{\nu+1}(x,y) &\lesi x\f{e^{-t/2}}{\sqrt t} \exp\Big(-\f{|x-y|^2}{ct}\Big) \Big(1+\f{\sqrt t}{x}\Big)^{-(\nu+1/2)}\\
	&\lesi  \f{1}{t} \exp\Big(-\f{|x-y|^2}{ct}\Big) \Big(1+\f{\sqrt t}{x}\Big)^{-(\nu+3/2)}
\end{aligned}
\]
and
\[
\begin{aligned}
	\f{|y-x|}{t}p_t^{\nu+1}(x,y) 
	&\lesi   |y-x| \f{e^{-t/2}}{\sqrt t} \exp\Big(-\f{|x-y|^2}{ct}\Big) \Big(1+\f{\sqrt t}{x}\Big)^{-(\nu+3/2)}\\
	&\lesi   \f{1}{t} \exp\Big(-\f{|x-y|^2}{2ct}\Big) \Big(1+\f{\sqrt t}{x}\Big)^{-(\nu+3/2)}.
\end{aligned}
\]			
This ensures \eqref{eq- stronger estimate delta k p} for $k=1$ and $t>1$.

This completes the proof of \eqref{eq- stronger estimate delta k p} for $k=1$.

\bigskip

$\bullet$ Suppose that \eqref{eq- stronger estimate delta k p} is true for all  $k=1, 2, \ldots, 2\ell +1$ for some $\ell\in \mathbb N$. We need to prove the estimate \eqref{eq- stronger estimate delta k p} for $k=2\ell+3$. Indeed, we have
\[
\delta_\nu^{2\ell+3}e^{-t\mathcal L_\nu} =\delta_\nu^{2\ell+2}e^{-\f{t}{2}(\mathcal L_{\nu+1}+2)}\circ \delta_\nu  e^{-\f{t}{2}\mathcal L_\nu}.
\]
Since the kernel of $\delta_\nu  e^{-\f{t}{2}\mathcal L_\nu}$ satisfies \eqref{eq- stronger estimate delta k p}, it suffices to show that  the kernel of $\delta_\nu^{2\ell+2}e^{-\f{t}{2}\mathcal L_{\nu+1}}$ has a Gaussian upper bound. Indeed, using \eqref{eq- del nu and del nu + 1},
\[
\delta_\nu^{2\ell+2}e^{-\f{t}{2}\mathcal L_{\nu+1}}=\delta_{\nu+1}^{2\ell+2}e^{-\f{t}{2}\mathcal L_{\nu+1}}+\f{2\ell+2}{x}\delta_{\nu+1}^{2\ell+1}e^{-\f{t}{2}\mathcal L_{\nu+1}}.
\]		
This, together with \eqref{eq- delta k pt} and the inductive hypothesis for $k=2\ell+1$, implies that 
\[
\begin{aligned}
	|\delta_\nu^{2\ell+2}p_{t/2}^{\nu+1}(x,y)|\lesi \ &\f{1}{t^{(2\ell+3)/2}} \exp\Big(-\f{|x-y|^2}{2ct}\Big) \Big(1+\f{\sqrt t}{x}\Big)^{-(\nu+3/2)}\\
	&+ \f{1}{x}\f{1}{t^{(2\ell+2)/2}} \exp\Big(-\f{|x-y|^2}{2ct}\Big) \Big(1+\f{\sqrt t}{x}\Big)^{-(\nu+5/2)}\\
	\lesi \ &\f{1}{t^{(2\ell+3)/2}} \exp\Big(-\f{|x-y|^2}{2ct}\Big) \Big(1+\f{\sqrt t}{x}\Big)^{-(\nu+3/2)}.
\end{aligned}
\]		
		This completes our proof.

\end{proof}

\begin{prop}\label{prop1- delta partial}
	Let $\nu>-1$. Then for each $k,\ell\in \mathbb N$, we have
	\[
	|\delta_\nu^k\partial_t^\ell p_t^\nu(x,y)|\lesi \f{1}{t^{(k+\ell+1)/2}}\exp\Big(-\f{|x-y|^2}{ct}\Big) \Big(1+\f{\sqrt t}{x}\Big)^{\gamma_{\nu}} \Big(1+\f{\sqrt t}{y}\Big)^{\gamma_{\nu}}
	\]
	for all $t>0$ and $x,y\in \mathbb R_+$.
\end{prop}
\begin{proof}
	Note that $\partial_t^\ell e^{-t\mathcal L_\nu} = (-1)^\ell \mathcal L_\nu^ke^{-t\mathcal L_\nu}$. Hence, we can write
	\[
	\delta_\nu^k\partial_t^\ell e^{-t\mathcal L_\nu} =(-1)^\ell \delta_\nu^k e^{-\f{t}{2}\mathcal L_\nu}\circ  \mathcal L_\nu e^{-\f{t}{2}\mathcal L_\nu} = 2^\ell  \delta_\nu^k e^{-\f{t}{2}\mathcal L_\nu}\circ \partial_t^\ell e^{-\f{t}{2}\mathcal L_\nu},
	\]
	which implies
	\begin{equation}\label{eq- proof of prop 1}
	|\delta_\nu^k\partial_t^\ell p_t^\nu(x,y)|\lesi \int_{\mathbb R_+} |\delta_\nu^kp_{t/2}^\nu(x,z)| |\partial_t^\ell p_{t/2}^\nu(z,y)|dz.
	\end{equation}
	This, together with Proposition \ref{thm-delta k pnu}, \eqref{eq-partial pt} and Lemma \ref{lem-elementary lemma}, deduces the desired estimate.
	
	This complete our proof.
\end{proof}
\medskip
\begin{prop}\label{prop2- delta partial}
	Let $\nu>-1$. Then for each $\ell\in \mathbb N$, we have
	\[
	|\delta^*_\nu\partial_t^\ell p_t^{\nu+1}(x,y)|\lesi \f{1}{t^{(\ell+2)/2}}\exp\Big(-\f{|x-y|^2}{ct}\Big) \Big(1+\f{\sqrt t}{x}\Big)^{\gamma_{\nu}}
	\]
	for all $t>0$ and $x,y\in \mathbb R_+$.
\end{prop}
\begin{proof}
	Similarly to the proof of Proposition \ref{prop1- delta partial}, we have
	\[
	|\delta^*_\nu\partial_t^\ell p_t^{\nu+1}(x,y)|\lesi \int_{\mathbb R_+} |\delta_\nu^* p_{t/2}^{\nu+1}(x,z)| |\partial_t^\ell p_{t/2}^{\nu+1}(z,y)|dz.
	\]
	This, in combination with \eqref{eq- delta* pt},  \eqref{eq-partial pt} and Lemma \ref{lem-elementary lemma}, yields the desired estimate.
	
	This completes our proof.
\end{proof}
\begin{prop}\label{prop3- delta partial}
	Let $\nu>-1/2$. For \textbf{ any odd} $k\in \mathbb N$ and $\ell\in \mathbb N$, we have
	\begin{equation*}
		|\delta_{\nu}^k\partial_t^\ell p_t^{\nu}(x,y)|
		\lesi  \f{1}{t^{(k+\ell+1)/2}} \exp\Big(-\f{|x-y|^2}{ct}\Big) \Big(1+\f{\sqrt t}{x}\Big)^{-(\nu+3/2)}
	\end{equation*}
	for all $x,y\in \mathbb R_+$ and $t>0$.
	
	Consequently, for any \textbf{ any odd} $k\in \mathbb N$,
	\[
	\f{1}{x}|\delta_{\nu}^k \partial_t^\ell p_t^{\nu}(x,y)|
	\lesi  \f{1}{t^{(k+\ell+2)/2}} \exp\Big(-\f{|x-y|^2}{ct}\Big) \Big(1+\f{\sqrt t}{x}\Big)^{-(\nu+1/2)}
	\]
	for all $x,y\in \mathbb R_+$ and $t>0$.
	
\end{prop}
\begin{proof}
	The second inequality is a  consequence of the first one, while the first follows directly from \eqref{eq- proof of prop 1}, \eqref{eq-partial pt} and Proposition \ref{prop-2-stronger estimate on delta k }.
	
	This completes our proof.
\end{proof}

\bigskip

\noindent\underline{\textbf{The case $n\ge 2$}}

\bigskip


The following theorem is a consequence of Theorem \ref{thm-delta k pnu} and \eqref{eq- prod ptnu}.
\begin{thm}\label{thm1- kernel est  n ge 2}
Let $\nu\in (-1,\vc)^n$ and $k\in \mathbb N^k$. Then we have
\[
|\delta_\nu^kp_t^\nu(x,y)|\lesi \f{1}{t^{(n+|k|)/2}}\exp\Big(-\f{|x-y|^2}{ct}\Big)\prod_{j=1}^n \Big(1+\f{\sqrt t}{y_j}\Big)^{\gamma_{\nu_j}}\prod_{j: k_j \ {\rm is\ even} } \Big(1+\f{\sqrt t}{x_j}\Big)^{\gamma_{\nu_j}}
\]
for all $t>0$ and $x,y\in \Rn_+$.
\end{thm}

Theorems \ref{thm1s- kernel est  n ge 2}, \ref{thm2- kernel est  n ge 2} and Theorem \ref{thm3- kernel est  n ge 2} below are consequences of Propositions \ref{prop1- delta partial}, \ref{prop2- delta partial}, \ref{prop3- delta partial} and \eqref{eq- prod ptnu}.
\begin{thm}\label{thm1s- kernel est  n ge 2}
	Let $\nu\in (-1,\vc)^n$, $k\in \mathbb N^k$ and $\ell\in \mathbb N$. Then we have
	\[
	|\delta_\nu^k\partial_t^\ell p_t^\nu(x,y)|\lesi \f{1}{t^{(n+|k|+\ell)/2}}\exp\Big(-\f{|x-y|^2}{ct}\Big)\prod_{j=1}^n  \Big(1+\f{\sqrt t}{x_j}\Big)^{\gamma_{\nu_j}} \Big(1+\f{\sqrt t}{y_j}\Big)^{\gamma_{\nu_j}} 
	\]
	for all $t>0$ and $x,y\in \Rn_+$.
\end{thm}

\begin{thm}\label{thm2- kernel est  n ge 2}
	Let $\nu\in (-1,\vc)^n$, $k\in\{0,1\}^n$ and $\ell\in \mathbb N$. Then we have
	\[
	|(\delta^*_\nu)^k\partial_t^\ell p_t^{\nu+k}(x,y)|\lesi \f{1}{t^{(n+|k|+\ell)/2}}\exp\Big(-\f{|x-y|^2}{ct}\Big)\prod_{j=1}^n \Big(1+\f{\sqrt t}{x_j}\Big)^{\gamma_{\nu_j}}\prod_{j: k_j=0 } \Big(1+\f{\sqrt t}{y_j}\Big)^{\gamma_{\nu_j}}
	\]
	for all $t>0$ and $x,y\in \Rn_+$.
\end{thm}

\begin{thm}\label{thm3- kernel est  n ge 2}
	Let $\nu\in (-1,\vc)^n$, $k=(k_1,\ldots, k_n)\in \mathbb N^k$ and $\ell\in \mathbb N$. Assume that all odd entries of $k$ are greater than $-1/2$. Then we have
	\[
	\begin{aligned}
		\Big|\prod_{j: k_j \ \text{\rm is odd}}  \f{1}{x_j} \delta_{\nu_j}^{k_j}\prod_{j: k_j \ \text{\rm is even}}& \delta_{\nu_j}^{k_j}\partial_t^\ell p_t^{\nu}(x,y)\Big|\\
		&\lesi \f{1}{t^{(n+|k|+\ell)/2}}\exp\Big(-\f{|x-y|^2}{ct}\Big)\prod_{j=1}^n \Big(1+\f{\sqrt t}{y_j}\Big)^{\gamma_{\nu_j}}\prod_{j: k_j {\rm is\ even} } \Big(1+\f{\sqrt t}{x_j}\Big)^{\gamma_{\nu_j}}
	\end{aligned}
	\]
	for all $t>0$ and $x,y\in \Rn_+$.
\end{thm}

\subsection{$L^p$-boundedness of some integral operators}
 
\begin{lem}\label{lem-314}
	Let $\alpha\in (0,1)$. Define
	\[
	T_\alpha f(x) = \int_{x/2}^{2x} K_\alpha(x,y)f(y)dy, 
	\]
	where 
	\[
	K_\alpha(x,y)=\f{1}{x}\Big(\f{x}{|x-y|}\Big)^\alpha
	\]
	Then $T_\alpha$ is bounded on $L^p(\mathbb R_+)$ for all $\f{1}{1-\alpha}<p<\vc$.
\end{lem}
\begin{proof}
	Fix $p\in (\f{1}{1-\alpha},\vc)$ and $\f{1}{1-\alpha}<r<p$. Then we have $\alpha r'<1$. This, together with  H\"older's inequality, implies 
	\[
	\begin{aligned}
		T_\alpha f(x)&\lesi \int_{x/2}^{2x} \f{1}{x}\Big(\f{x}{|x-y|}\Big)^{\alpha} |f(y)|dy\\
		&\lesi \Big[\f{1}{x}\int_{x/2}^{2x}  |f(y)|^rdy\Big]^{1/r}\Big[\f{1}{x}\int_{x/2}^{2x} \Big(\f{x}{|x-y|}\Big)^{\alpha r'}  dy\Big]^{1/r'}\\
		&\lesi \Big[\f{1}{x}\int_{x/2}^{2x}  |f(y)|^rdy\Big]^{1/r}\\
		&\lesi \big[\mathcal M(|f|^r)(x)\big]^{1/r}.		
	\end{aligned}
	\]
	Since $\mathcal M$ is bounded on $L^q$ for any $q>1$, the above inequality implies that $T_\alpha$ is bounded on $L^p(\mathbb R_+)$.
	
	This completes our proof.
\end{proof}

\begin{prop}\label{prop1-boundedness}
Let $\alpha,\beta \in [0,1)$ so that $\alpha + \beta <1$. Let  the operator $\textbf{S}_{\alpha,\beta}$ defined by 
$$
\textbf{S}_{\alpha,\beta}f(x):=\sup_{t>0} \Big|\int_{\mathbb R_+}S_{\alpha,\beta}^t(x,y) |f(y)|dy\Big|,
$$
where
\[
S^t_{\alpha,\beta}(x,y) = \f{1}{\sqrt t}\exp\Big(-\f{|x-y|^2}{ct}\Big)\Big(1+\f{\sqrt t}{x}\Big)^\alpha\Big(1+\f{\sqrt t}{y}\Big)^\beta.
\]
Then $\textbf{S}_{\alpha,\beta}$ is bounded on $L^p(\mathbb R_+)$ for all $\f{1}{1-\beta}<p<\f{1}{\alpha}$.
\end{prop}
\begin{proof}
We first claim that 		 
	\begin{equation}\label{eq- bound of S a b t}
	S^t_{\alpha,\beta}(x,y) \lesi \f{1}{\sqrt t}\exp\Big(-\f{|x-y|^2}{ct}\Big)+\f{1}{x} \chi_{\{y/2< x<2y\}}+ \f{1}{x}\Big(\f{x}{y}\Big)^{\beta}\chi_{\{2y\le  x\}} +\f{1}{y}\Big(\f{y}{x}\Big)^{\alpha}\chi_{\{x\le y/2\}}
	\end{equation}
	for  $t>0$ and $x,y>0$.
	
\textbf{Case 1: $y/2<x<2y$} 

If $x>\sqrt t$, then we have
\[
1+\f{\sqrt t}{x} \sim 1+\f{\sqrt t}{y}  \sim 1.
\]
Hence,
\[
S^t_{\alpha,\beta}(x,y) \sim \f{1}{\sqrt t}\exp\Big(-\f{|x-y|^2}{ct}\Big).
\]
If $x\le \sqrt t$, then in this case  
\[
1+\f{\sqrt t}{x} \sim 1+\f{\sqrt t}{y} \sim \f{\sqrt t}{x}.
\]
Hence,
\[
\begin{aligned}
	S^t_{\alpha,\beta}(x,y)&\sim \f{1}{\sqrt t}\exp\Big(-\f{|x-y|^2}{ct}\Big) \Big(\f{\sqrt t}{x}\Big)^{\alpha+\beta} \\
	&\lesi \f{1}{\sqrt t}  \Big(\f{\sqrt t}{x}\Big)^{\alpha+\beta} \\
	&\lesi \f{1}{x}.
\end{aligned}
\]
Consequently,
\[
S^t_{\alpha,\beta}(x,y)\chi_{\{y/2<x<2y\}}\lesi \Big[\f{1}{\sqrt t}\exp\Big(-\f{|x-y|^2}{ct}\Big)+\f{1}{x}\Big]\chi_{\{y/2<x<2y\}}
\]

\medskip

\textbf{Case 2: $x\ge 2y$.} In this case $|x-y|\sim x$ 

If $\sqrt t <2y\le x $, then 
\[
S^t_{\alpha,\beta}(x,y) \sim \f{1}{\sqrt t}\exp\Big(-\f{x^2}{ct}\Big)\lesi  \f{1}{x}.
\]
If $2y\le \sqrt t\le x$, then
\[
\begin{aligned}
	S^t_{\alpha,\beta}(x,y) &\sim \f{1}{\sqrt t}\exp\Big(-\f{x^2}{ct}\Big) \Big(\f{\sqrt t}{y}\Big)^\beta\\
	&\lesi \f{1}{x}\Big(\f{x}{y}\Big)^\beta.
\end{aligned}
\]
If $2y \le x<\sqrt t$, then
\[
\begin{aligned}
	S^t_{\alpha,\beta}(x,y) &\sim \f{1}{\sqrt t}  \Big(\f{\sqrt t}{x}\Big)^\alpha\Big(\f{\sqrt t}{y}\Big)^\beta\\
	&\lesi \f{1}{x}\Big(\f{x}{y}\Big)^\beta.
\end{aligned}
\]
Consequently,
\[
S^t_{\alpha,\beta}(x,y)\chi_{\{x\ge 2y\}} \lesi \f{1}{x}\Big(\f{x}{y}\Big)^\beta \chi_{\{x\ge 2y\}}.
\]

\medskip

\textbf{Case 3: $x\le  y/2$.} In this case, $y\ge 2x$ and  $|x-y|\sim y$. 

If $\sqrt t <2x\le y $, then 
\[
S^t_{\alpha,\beta}(x,y) \sim \f{1}{\sqrt t}\exp\Big(-\f{y^2}{ct}\Big)\lesi  \f{1}{y}.
\]
If $2x\le \sqrt t\le y$, then
\[
\begin{aligned}
	S^t_{\alpha,\beta}(x,y) &\sim \f{1}{\sqrt t}\exp\Big(-\f{y^2}{ct}\Big) \Big(\f{\sqrt t}{x}\Big)^\alpha\\
	&\lesi \f{1}{y}\Big(\f{y}{x}\Big)^\alpha.
\end{aligned}
\]
If $2x \le y<\sqrt t$, then
\[
\begin{aligned}
	S^t_{\alpha,\beta}(x,y) &\sim \f{1}{\sqrt t}  \Big(\f{\sqrt t}{y}\Big)^\alpha\Big(\f{\sqrt t}{y}\Big)^\beta\\
	&\lesi \f{1}{y}\Big(\f{y}{x}\Big)^\alpha.
\end{aligned}
\]
Consequently,
\[
S^t_{\alpha,\beta}(x,y)\chi_{\{x\le y/2\}} \lesi \f{1}{y}\Big(\f{y}{x}\Big)^\alpha \chi_{\{x\le y/2\}}.
\]
These estimates confirm \eqref{eq- bound of S a b t}.

We now turn to prove the boundedness of $\textbf{S}_{\alpha,\beta}$. Fix $\f{1}{1-\beta}<p<\f{1}{\alpha}$. Due to \eqref{eq- bound of S a b t},
\[
\begin{aligned}
	\textbf{S}_{\alpha,\beta}f(x) &\lesi \sup_{t>0} \int_{\mathbb R_+}\f{1}{\sqrt t}\exp\Big(-\f{|x-y|^2}{ct}\Big)|f(y)|dy +\f{1}{x}\int_{x/2}^{2x}|f(y)|dy \\
	& \ \ \ + \f{1}{x}\int_{0}^{x/2}\Big(\f{x}{y}\Big)^{\beta}|f(y)|dy + \int_{2x}^\vc \f{1}{y}\Big(\f{y}{x}\Big)^{\alpha}|f(y)|dy\\
	&=:\textbf{S}^1_{\alpha,\beta}f(x) +\textbf{S}^2_{\alpha,\beta}f(x) +\textbf{S}^3_{\alpha,\beta}f(x) +\textbf{S}^4_{\alpha,\beta}f(x).
\end{aligned}
\]
The standard argument shows that
\[
\textbf{S}^1_{\alpha,\beta}f(x)\lesi \mathcal Mf(x),
\] 
where we recall that $\mathcal M$ is the Hardy-Littlewood maximal function on $\mathbb R_+$.

Obviously,
\[
\textbf{S}^2_{\alpha,\beta}f(x)\lesi \mathcal Mf(x).
\]
Fix $p_1$ such that $\f{1}{1-\beta}<p_1<p<\f{1}{\alpha}$, which implies $\beta p_1'<1$. By the H\"older inequality and the fact $\beta p_1'<1$,
\[
\begin{aligned}
	\textbf{S}^3_{\alpha,\beta}f(x)&\lesi \Big(\f{1}{x} \int_0^{2x}|f(y)|^{p_1}dy\Big)^{1/p_1}\Big(\f{1}{x} \int_0^{2x} \Big(\f{x}{y}\Big)^{\beta p_1'} dy\Big)^{1/p_1'}\\
	&\lesi \Big(\f{1}{x} \int_0^{2x}|f(y)|^{p_1}dy\Big)^{1/p_1}\\
	&\lesi  \mathcal M_{p_1}f(x).
\end{aligned}
\]
These three estimates and Lemma \ref{Lem-maximalfunction} imply that $\textbf{S}^1_{\alpha,\beta}, \textbf{S}^2_{\alpha,\beta}$ and $\textbf{S}^3_{\alpha,\beta}$ are bounded on $L^p(\mathbb R_+)$.

It remains to show that $\textbf{S}^4_{\alpha,\beta}$ are bounded on $L^p(\mathbb R_+)$. Fix $p_2$  such that $\f{1}{1-\alpha}<p_2<p'<\f{1}{\beta}$. For $g\in L^{p'}(\mathbb R_+)$, we have
\[
\begin{aligned}
	\langle \textbf{S}^4_{\alpha,\beta}f, g\rangle &\le \int_0^\vc |g(x)|\int_{2x}^\vc \f{1}{y}\Big(\f{y}{x}\Big)^{\alpha}|f(y)|dydx\\
	&\le \int_0^\vc |f(y)|\int_{0}^{y/2} \f{1}{y} |g(x)| \Big(\f{y}{x}\Big)^{\alpha}dxdy.
\end{aligned}
\]
Similarly to $\textbf{S}^3_{\alpha,\beta}$, we have
\[
\int_{0}^{y/2} \f{1}{y} |g(x)| \Big(\f{y}{x}\Big)^{\alpha}dx\lesi \mathcal M_{p_2}f(x),
\]
which implies
\[
\langle \textbf{S}^4_{\alpha,\beta}f, g\rangle \lesi \int_0^\vc |f(y)|]\mathcal M_{p_2}f(y) dy =\langle f, \mathcal M_{p_2}f \rangle.
\]
This, together with Lemma \ref{Lem-maximalfunction}, implies that the operator $\textbf{S}^4_{\alpha,\beta}$ is bounded on $L^p(\mathbb R_+)$.

This completes our proof.
\end{proof}
\begin{prop}\label{prop2 - boundedness}
	Let $\nu>-1$ and $k\in \mathbb N$. Then for any $\epsilon>0$, we have
	\[
	\begin{aligned}
		\int_0^\vc \f{1}{x} &\exp\Big(-\f{|x-y|^2}{ct}\Big)\Big(1+\f{\sqrt t}{x}\Big)^{-(\nu+3/2)}\Big(1+\f{\sqrt t}{y}\Big)^{-(\nu+3/2)}\f{dt}{t}\\
	&\lesi \Big[\f{1}{x} + \f{1}{x}\Big(\f{x}{|x-y|}\Big)^{\epsilon}\Big]\chi_{\{y/2<x<2y\}} + \f{1}{x}\chi_{\{x\ge 2y\}} + \f{1}{y}\Big(\f{y}{x}\Big)^{-\nu-1/2}\chi_{\{y\ge 2x\}}.
		\end{aligned}\]
	Consequently, the operator
	\[
	f\mapsto \int_0^\vc\int_0^\vc \f{1}{x} \exp\Big(-\f{|x-y|^2}{ct}\Big)\Big(1+\f{\sqrt t}{x}\Big)^{-(\nu+3/2)}\Big(1+\f{\sqrt t}{y}\Big)^{-(\nu+3/2)}|f(y)|\f{dt}{t}dy
	\]
	is bounded on $L^p(\mathbb R_+)$ for $1<p<\f{1}{(-\nu-1/2)\vee 0}$ with convention $\f{1}{0}=\vc$.
\end{prop}
\begin{proof}
	Once the inequality has been obtained, the boundedness of the operator  
	\[
	f\mapsto \int_0^\vc\int_0^\vc \f{1}{x} \exp\Big(-\f{|x-y|^2}{ct}\Big)\Big(1+\f{\sqrt t}{x}\Big)^{-(\nu+3/2)}\Big(1+\f{\sqrt t}{y}\Big)^{-(\nu+3/2)}|f(y)|\f{dt}{t}dy
	\]
	can be done similarly to the proof of Proposition \ref{prop1-boundedness} and Lemma \ref{lem-314}; hence, we omit the details. For this reason, we need only to prove the  inequality. Define
	\[
	F(x,y)=\int_0^\vc \f{1}{x} \exp\Big(-\f{|x-y|^2}{ct}\Big)\Big(1+\f{\sqrt t}{x}\Big)^{-(\nu+3/2)}\Big(1+\f{\sqrt t}{y}\Big)^{-(\nu+3/2)}\f{dt}{t}.
	\]
We have the following three cases.

	\textbf{Case 1: $y/2<x<2y$} 
	
We have
	\[
	\begin{aligned}
		F(x,y) &\lesi \int_0^{x^2}\f{1}{x} \exp\Big(-\f{|x-y|^2}{ct}\Big)  \f{dt}{t}+\int_{x^2}^\vc \f{1}{x} \Big(\f{\sqrt t}{x}\Big)^{-(\nu+3/2)}\Big(\f{\sqrt t}{y}\Big)^{-(\nu+3/2)}\f{dt}{t}\\
		&\lesi \f{1}{x}\Big(\f{x}{|x-y|}\Big)^{\epsilon}+\f{1}{x}.
	\end{aligned}
	\]
	
	\textbf{Case 2: $x\ge 2y$}
	
	In this case, $|x-y|\sim x$. Hence,
	\[
	\begin{aligned}
		F(x,y)&\lesi \int_0^{x^2}\f{1}{x} \exp\Big(-\f{x^2}{ct}\Big)  \f{dt}{t}+\int_{x^2}^\vc \f{1}{x} \Big(\f{\sqrt t}{x}\Big)^{-(\nu+3/2)} \f{dt}{t}\\
		&\lesi \f{1}{x}.
	\end{aligned}
	\]
	
	\textbf{Case 3: $x\le y/2$, i.e., $ y\ge 2x$}
	
	In this case, $|x-y|\sim x$. It follows that 
	\[
	\begin{aligned}
		F(x,y)&\lesi \int_0^{y^2}\f{1}{x} \exp\Big(-\f{x^2}{ct}\Big) \Big(\f{\sqrt t}{x}\Big)^{-(\nu+3/2)} \f{dt}{t}+\int_{y^2}^\vc \f{1}{x} \Big(\f{\sqrt t}{x}\Big)^{-(\nu+3/2)} \f{dt}{t}\\
		&\lesi \f{1}{x}\Big(\f{y}{x}\Big)^{-(\nu+3/2)}=\f{1}{y}\Big(\f{y}{x}\Big)^{-(\nu+1/2)}.
	\end{aligned}
	\]

	This completes our proof.
\end{proof}

Let $\sigma, \beta\in [0,1/2)$. For each $t>0$, define
\begin{equation}\label{eq-Tt}
	T_{t,\beta,\sigma}f(x)=\int_{\Rn_+}T_{t,\beta,\sigma}(x,y)f(y)dy,
\end{equation}
where
\begin{equation}\label{eq-Tt kernel}
	T_{t,\beta,\sigma}(x,y)= \f{1}{t^{n/2}}\exp\Big(-\f{|x-y|^2}{ct}\Big)\prod_{j=1}^n\Big(1+\f{\sqrt t}{x_j}\Big)^\beta\Big(1+\f{\sqrt t}{y_j}\Big)^\sigma.
\end{equation}

The following lemma is taken from \cite[Lemma 3.7]{B2}.
\begin{lem}\label{lem Lpq}
	Let $\sigma\in (0,1/2)$ and $\beta>0$ such that $\sigma+\beta<1$. Let $T_{t,\beta,\sigma}$ be the linear operator defined by \eqref{eq-Tt} and \eqref{eq-Tt kernel}. Then for any $\f{1}{1-\sigma}<p\le q<\f{1}{\beta}$ and any ball $B\subset \mathbb R^n_+$, we have
	\begin{equation}
		\label{eq- eq1 Tt off diagonal}
		\|T_{t,\beta,\sigma}\|_{p\to q}\lesi t^{-\f{n}{2}(\f{1}{p}-\f{1}{q})},
	\end{equation}
	and
	\begin{equation}
		\label{eq- eq2 Tt off diagonal}
		\|T_{t,\beta,\sigma}\|_{L^p(B)\to L^q(S_j(B))}+\|T_{t,\beta,\sigma}\|_{L^p(S_j(B))\to L^q(B)}\lesi t^{-\f{n}{2}(\f{1}{p}-\f{1}{q})}\exp\Big(-\f{(2^jr_B)^2}{ct}\Big), \ \ j\ge 2,
\end{equation}
where recall that $S_j(B)=2^{j}B\backslash 2^{j-1}(B)$ for $j\ge 1$.
\end{lem}

\section{Weighted estimates for the higher-order Riesz transforms}
This section is dedicated to proving Theorem \ref{mainthm-genral case}. The proof is quite long and complicated and will be done through several steps. We first prove a special case of Theorem \ref{mainthm-genral case}.  


\begin{thm}\label{mainthm-all indices are 0 or 1}
	Let $\nu\in (-1,\vc)^n$ and $a\ge 0$. Suppose that $k=(k_1,\ldots, k_n)\in \mathbb N^n$ such that $k_1, \ldots, k_n \in \{0,1\}$. Then the Riesz transform $\delta_\nu^k (\mathcal L_\nu+a)^{-|k|/2}$ is bounded on $L^p_w(\Rn_+)$ for all $\f{1}{1-\gamma_\nu}<p<\f{1}{\gamma_{\nu+\sigma(k)}}$ and $w\in A_{p(1-\gamma_\nu)}\cap RH_{(\f{1}{p\gamma_{\nu+\sigma(k)}})'}$.
\end{thm}

\begin{proof}
	Note that in this case $k=\sigma(k)$.

	It suffices to prove the theorem for $a=0$ since the proof for the case $a>0$ can be done similarly. 
	
	We first prove that the Riesz transform $\delta_\nu^k  \mathcal L_\nu ^{-|k|/2}$ is bounded on $L^2(\Rn_+)$. Assume that the first $j$ entries of $k$ are equal to $1$ and other entries are $0$. In this case, $|k|=j$. By \eqref{eq- delta and eigenvector} we have
	\[
	\begin{aligned}
		\delta_\nu^k\mathcal{L}_{\nu}^{-|k|/2}f& = \sum_{\alpha\in \mathbb{N}^{n}} (4|\alpha| + 2|\nu| + 2n)^{-k/2} \langle f, \varphi_{\alpha}^\nu\rangle \delta_\nu^k\varphi_{\alpha}^\nu\\
		& = \sum_{\alpha\in \mathbb{N}^{n}} \f{-2^{|k|}\sqrt{\alpha_1\ldots \alpha_j}}{ {(4|\alpha| + 2|\nu| + 2n)^{|k|/2}}} \langle f, \varphi_{\alpha}^\nu\rangle \varphi_{\alpha-k}^{\nu+k},
	\end{aligned} 
	\]
	which implies
	\[
	\|\delta_\nu^k\mathcal{L}_{\nu}^{-|k|/2}f\|_2^2\le \sum_k | \langle f, \varphi_{k}^\nu\rangle|^2 =\|f\|_2^2. 
	\]
	Hence, the Riesz transform $\delta_\nu^k  \mathcal L_\nu ^{-|k|/2}$ is bounded on $L^2(\Rn_+)$.
	
	 We now prove that $\delta_\nu^k\mathcal L^{-|k|/2}_\nu$ is bounded on $\frac{1}{1-\gamma_\nu}<p\le 2$.   Fix $m\in \mathbb N$ such that $2m-1>n$. In the light of Theorem \ref{BZ-thm}, it suffices to prove that for any $\f{1}{1-\gamma_\nu}<p_0<2$,  
	\begin{equation}\label{eq1-proof of mainthm}
		\begin{aligned}
			\Big(\fint_{S_j(B)}|\delta_\nu^k\mathcal L^{-|k|/2}_\nu(I-e^{-r_B^2\mathcal L_\nu})^mf|^{2} dx \Big)^{1/2}\lesi 2^{-j(2m-1)}\Big(\fint_B |f|^{2}dx\Big)^{1/2}
		\end{aligned}
	\end{equation}
	and
	\begin{equation}\label{eq2-proof of mainthm}
		\begin{aligned}
			\Big(\fint_{S_j(B)}|[I-(I-e^{-r_B^2\mathcal L_\nu})^m]f|^{2}dx\Big)^{1/2}\leq
			2^{-j(n+1)}\Big(\fint_B |f|^{p_0}dx\Big)^{1/p_0} 
		\end{aligned}
	\end{equation}
	for all $j\ge 2$ and $f\in C^\vc(\mathbb R^n_+)$ supported in $B$.

	Note that	\[
	I-(I-e^{-r_B^2\mathcal L_\nu})^m =\sum_{i=1}^m c_ie^{-ir_B^2\mathcal L_\nu}, 
	\]
	where $c_i$ are constants.
	
	This, along with Theorem \ref{thm-delta k pnu} and Lemma \ref{lem Lpq}, implies \eqref{eq2-proof of mainthm}. 
	
	It remains to prove \eqref{eq1-proof of mainthm}. To do this, recall that
	\[
	\delta_\nu^k\mathcal L^{-|k|/2}_\nu = c\int_0^\vc  t^{|k|/2} \delta_\nu^k e^{-t\mathcal L_\nu} \f{dt}{t}
	\]
	and
	\[
	(I-e^{-r_B^2\mathcal L_\nu})^m=\int_{[0,r_B^2]^m}\mathcal L_\nu^m e^{-(s_1+\ldots+s_m)\mathcal L_\nu}d\vec{s},
	\]
	where $d\vec{s}=ds_1\ldots ds_m$.
	
	Hence, we can write
	\[
	\delta_\nu^k\mathcal L^{-|k|/2}_\nu(I-e^{-r_B^2\mathcal L_\nu})^mf =c\int_0^\vc\int_{[0,r_B^2]^m} t^{|k|/2} \delta_\nu^k \mathcal L_\nu^{m}e^{-(t+s_1+\ldots+s_m)\mathcal L_\nu} f\,d\vec{s}\,\f{dt}{t}. 
	\]
	Applying  Theorem \ref{thm1s- kernel est  n ge 2} and Lemma \ref{lem Lpq}, we obtain
	\[
	\begin{aligned}
		\|\delta_\nu^k\mathcal L^{-|k|/2}_\nu&(I-e^{-r_B^2\mathcal L_\nu})^mf\|_{L^{2}(S_j(B))}\\&\lesi \int_0^\vc\int_{[0,r_B^2]^m} t^{|k|/2} \big\| \delta_\nu^k \mathcal L_\nu^{m}e^{-(t+s_1+\ldots+s_m)\mathcal L_\nu} f\big\|_{L^{2}(S_j(B))}\,d\vec{s}\,\f{dt}{t}\\
		&\lesi \|f\|_{L^{2}(B)}\int_0^\vc\int_{[0,r_B^2]^m} \f{ t^{|k|/2}}{(t+|\vec s|)^{m+|k|/2}}\exp\Big(-\f{(2^jr_B)^2}{c(t+|\vec s|)}\Big)  d\vec{s}\f{dt}{t}\\
		&\lesi \|f\|_{L^{2}(B)}\int_0^{r_B^2}\ldots \f{dt}{t}+\|f\|_{L^{2}(B)}\int_{r_B^2}^\vc\ldots \f{dt}{t}\\
		&=: E_1 + E_2,
	\end{aligned}
	\]
	where $|\vec s| =s_1+\ldots+s_m$.
	
	For the first term $E_1$, we have
	\[
	\begin{aligned}
		E_1&\lesi \|f\|_{L^{2}(B)}\int_0^{r_B^2}\int_{[0,r_B^2]^m} \f{  t^{|k|/2}}{(t+|\vec s|)^{m +|k|/2}} \Big( \f{t+|\vec s|}{(2^jr_B)^2}\Big)^{m  +|k|/2}  d\vec{s}\f{dt}{t}\\
		&\lesi \|f\|_{L^{2}(B)}\int_0^{r_B^2}\int_{[0,r_B^2]^m}  \f{\sqrt t}{(2^jr_B)^{2m+1}}   d\vec{s}\f{dt}{t}\\
		&\sim \|f\|_{L^{2}(B)} \f{r_B^{2m+1}}{(2^jr_B)^{2m+1}}\\
		&\sim  2^{-j(2m+1)}\|f\|_{L^{2}(B)} .
	\end{aligned}
	\]
	For the second term, note that $t+s_1+\ldots + s_m \sim t $ as $t\ge r_B^2$ and $s_1,\ldots, s_m\in [0,r_B^2]$. It follows that 
	\[
	\begin{aligned}
		E_2&\lesi \|f\|_{L^{2}(B)}\int_{r_B^2}^\vc\int_{[0,r_B^2]^m} \f{\sqrt t^{|k|/2}}{t^{m +|k|/2}} \Big( \f{t}{(2^jr_B)^2}\Big)^{m -1/2}  d\vec{s}\f{dt}{t}\\
		&\sim 2^{-j(2m-1)}\|f\|_{L^{2}(B)}.
	\end{aligned}
	\]
	Collecting the estimates of $E_1$ and $E_2$, we deduce to \eqref{eq1-proof of mainthm}.

	\bigskip
	
	Next we will prove that $\delta_\nu^k\mathcal L^{-|k|/2}_\nu$ is bounded on $L^p$ for $2< p < \f{1}{\gamma_{\nu+\sigma(k)}}$. By duality, it suffices to prove $(\delta_\nu^k\mathcal L^{-|k|/2}_\nu)^*:=\mathcal L_\nu^{-1/2}(\delta_\nu^*)^k$ is bounded on $L^p(\mathbb R^n_+)$ for $\f{1}{1-\gamma_{\nu+\sigma(k)}}<p_0\le 2$. To do this, fix $\f{1}{1-\gamma_{\nu+\sigma(k)}}<p_0\le 2$, $m\in \mathbb N$ such that $2m-1>n$, by Theorem \ref{BZ-thm}, we need to show that  
	\begin{equation}\label{eq1-proof of mainthm dual Riesz transform}
		\begin{aligned}
			\Big(\fint_{S_j(B)}|(\delta_\nu^k\mathcal L^{-|k|/2}_\nu)^*(I-e^{-r_B^2(\mathcal L_{\nu+\sigma(k)}+2|k|)})^mf|^2 dx \Big)^{1/2}\lesi 2^{-j(2m-1)}\Big(\fint_B |f|^2dx\Big)^{1/2}
		\end{aligned}
	\end{equation}
	and
	\begin{equation}\label{eq2-proof of mainthm dual}
		\begin{aligned}
			\Big(\fint_{S_j(B)}|[I-(I-e^{-r_B^2(\mathcal L_{\nu+\sigma(k)}+2|k|)})^m]f|^{2}dx\Big)^{1/2}\leq
			2^{-j(n+1)}\Big(\fint_B |f|^{p_0}dx\Big)^{1/p_0} 
		\end{aligned}
	\end{equation}
	for all $j\ge 2$ and $f\in C^\vc(\mathbb R^n_+)$ supported in $B$.
	
	The estimate \eqref{eq2-proof of mainthm dual} follows directly from Proposition \ref{prop- delta k pt nu -1/2} and Lemma \ref{lem Lpq} and the following expansion
	\[
	I-(I-e^{-r_B^2(\mathcal L_{\nu+\sigma(k)}+2|k|)})^m=\sum_{k=1}^mc_k e^{-kr_B^2(\mathcal L_{\nu+\sigma(k)}+2|k|)},
	\]
	where $c_k$ are constants.
	
	Hence, we need to prove \eqref{eq1-proof of mainthm dual Riesz transform}. We first note that for $f\in L^2(\mathbb R^n_+)$ and $\nu \in (-1,\vc)^n$ we have, for $t>0$,
	\begin{equation*}
		\mathcal L_{\nu}^{-|k|/2} e^{-t\mathcal L_{\nu} }f =\sum_{\ell\in \mathbb N^n} \f{e^{-t(4|\ell|+2|\nu|+2n)}}{(4|\ell|+2|\nu|+2n)^{|k|/2}}\langle f,\varphi^\nu_\ell\rangle \varphi^\nu_\ell.
	\end{equation*}
	Using this and \eqref{eq- delta and eigenvector}, by a simple calculation we come up with
	
	\begin{equation*}
	\begin{aligned}
		e^{-t\mathcal L_{\nu}} \mathcal L_\nu^{-|k|/2}(\delta^*)^k&(I-e^{-r_B^2(\mathcal L_{\nu+\sigma(k)}+2|k|}))^mf\\
		&= (\delta^*)^k e^{-t (\mathcal L_{\nu+\sigma(k)}+2|k|)}(\mathcal L_{\nu+\sigma(k)}+2|k|)^{-1/2}(I-e^{-r_B^2(\mathcal L_{\nu+\sigma(k)}+2|k|}))^mf.
	\end{aligned}	
	\end{equation*}
	At this stage, arguing similarly to the proof of \eqref{eq1-proof of mainthm} we come up with \eqref{eq1-proof of mainthm dual Riesz transform}.
	
	We have just proved that the Riesz transform $\delta_\nu^k(\mathcal L_\nu)^{-|k|/2} $ is bounded on $L^p(\Rn_+)$ for all $\f{1}{\gamma_\nu}<p<\f{1}{\gamma_{\nu+\sigma(k)}}$. To complete the proof, we need to prove the weighted estimate for the Riesz transform $\delta_\nu^k\mathcal L_\nu^{-|k|/2}$. By duality it suffices to show that $\mathcal L_\nu^{-|k|/2}(\delta_\nu^*)^k$ is bounded on $L^p_w(\mathbb R^n_+)$ for all $\f{1}{1-\gamma_{\nu+\sigma(k)}}<p<\f{1}{\gamma_\nu}$ and $w\in A_{p(1-\gamma_{\nu+\sigma(k)})}\cap RH_{(\f{1}{p\gamma_\nu})'}$. We now fix $\f{1}{1-\gamma_{\nu+\sigma(k)}}<p<\f{1}{\gamma_\nu}$ and $w\in A_{p(1-\gamma_{\nu+\sigma(k)})}\cap RH_{(\f{1}{p\gamma_\nu})'}$. By Lemma \ref{weightedlemma1}, we can find $\f{1}{1-\gamma_{\nu+\sigma(k)}}<p_0<p<q_0<\f{1}{\gamma_\nu}$ such that $w\in A_{p_0(1-\gamma_{\nu+\sigma(k)})}\cap RH_{(\f{1}{q_0\gamma_\nu})'}$. Similarly to \eqref{eq1-proof of mainthm dual Riesz transform} and \eqref{eq2-proof of mainthm dual},  by using Theorem \ref{thm1- kernel est  n ge 2} and Lemma \ref{lem Lpq} we also obtain
	\begin{equation*}
		\begin{aligned}
			\Big(\fint_{S_j(B)}|\mathcal L_\nu^{-|k|/2}(\delta^*)^k(I-e^{-r_B^2(\mathcal L_{\nu+\sigma(k)}+2|k|)})^mf|^{q_0} dx \Big)^{1/q_0}\lesi 2^{-j(2m-1)}\Big(\fint_B |f|^{q_0}dx\Big)^{1/q_0}
		\end{aligned}
	\end{equation*}
	and
	\begin{equation*}
		\begin{aligned}
			\Big(\fint_{S_j(B)}|[I-(I-e^{-r_B^2(\mathcal L_{\nu+\sigma(k)}+2|k|)})^m]f|^{{q_0}}dx\Big)^{1/{q_0}}\leq
			2^{-j(n+1)}\Big(\fint_B |f|^{p_0}dx\Big)^{1/p_0} 
		\end{aligned}
	\end{equation*}
	for all $j\ge 2$ and $f\in L^2(\mathbb R^n_+)$ supported in $B$.
	
	Hence, by Theorem \ref{BZ-thm}, $\mathcal L_\nu^{-|k|/2}(\delta^*)^k$ is bounded on $L^{p}_w(\mathbb R^n_+)$ with $w\in A_{p_0(1-\gamma_{\nu+\sigma(k)})}\cap RH_{(\f{1}{q_0\gamma_\nu})'}$.
	
	This completes our proof. 
\end{proof}

Secondly, we prove the following result.
\begin{thm}\label{mainthm-all even indices - weighted estimate}
	Let $\nu\in (-1,\vc)^n$. For $a\ge 0$ and $k \in \mathbb N^n$, the Riesz transform $\delta_\nu^k (\mathcal L_\nu+a)^{-|k|/2}$ is bounded on $L^p_w(\Rn_+)$ for all $\f{1}{1-\gamma_\nu}<p<\f{1}{\gamma_\nu}$ and $w\in A_{p(1-\gamma_\nu)}\cap RH_{(\f{1}{p\gamma_\nu})'}$.
\end{thm}
Note that Theorem \ref{mainthm-all even indices - weighted estimate} only gives the sharp estimate as in Theorem \ref{mainthm-genral case} in the case that all entries of $k$ are even. For other case, the estimates in Theorem \ref{mainthm-all even indices - weighted estimate} are not sharp. However, if all $\nu_j\ge -1/2$ for all $j=1,\ldots, n$, the estimates in Theorem  \ref{mainthm-all even indices - weighted estimate}  and Theorem \ref{mainthm-genral case} are the same since $\gamma_\nu=0$ in this case. This plays an important role in the proof of Theorem \ref{mainthm-genral case}.

\begin{proof}[Proof of Theorem \ref{mainthm-all even indices - weighted estimate}:]
We first  prove that the Riesz transform  $\delta_\nu^k (\mathcal L_\nu+a)^{-|k|/2}$ is bounded on $L^p(\Rn_+)$ for all $\f{1}{1-\gamma_\nu}<p<\f{1}{\gamma_\nu}$ by induction.
	
	\medskip
	
	\noindent $\bullet$ The statement is true for $|k|=1$ due to Theorem 
	\ref{mainthm-all indices are 0 or 1}.
	
	\medskip

	\noindent $\bullet$ Assume that  Riesz transform  $\delta_\nu^k (\mathcal L_\nu+a)^{-|k|/2}$ is bounded on $L^p_w(\Rn_+)$ for all $\f{1}{1-\gamma_\nu}<p<\f{1}{\gamma_\nu}$ for all $|k|\le \ell$ with  some $\ell\ge 1$ and for all $a\ge 0$. We need to prove the boundedness of $\delta_\nu^k (\mathcal L_\nu+a)^{-|k|/2}$ for $|k|=\ell+1$ and $a\ge 0$. For the sake of simplicity we might assume $a=0$ since the case $a>0$ can be done similarly with minor modifications. Indeed, due to Theorem \ref{mainthm-all indices are 0 or 1}, we consider only the case that at lease one entry of $k$ is greater than $1$. Without loss of generality, we might assume that   $k_1 \ge 2$. Then, using \eqref{eq- delta and eigenvector} we have
	\[
	\delta_\nu^k \mathcal L_\nu^{-|k|/2}= \delta_\nu^{k-e_1} \delta_{\nu_1}\mathcal L_\nu^{-(|k|-1)/2}\mathcal L_\nu^{-1/2}=\delta_\nu^{k-e_1}(\mathcal L_{\nu+e_1}+2)^{-(|k|-1)/2}\circ \delta_{\nu_1}\mathcal L_\nu^{-1/2},
	\]
	which, together with \eqref{eq- del nu and del nu + 1}, further implies
	\[
	\begin{aligned}
		\delta_\nu^k \mathcal L_\nu^{-|k|/2}	&=\delta_{\nu+e_1}^{k-e_1}(\mathcal L_{\nu+e_1}+2)^{-(|k|-1)/2}\circ \delta_{\nu_1}\mathcal L_\nu^{-1/2}+\f{1}{x_1}\delta_{\nu+e_1}^{k-2e_1}(\mathcal L_{\nu+e_1}+2)^{-(|k|-1)/2}\circ \delta_{\nu_1}\mathcal L_\nu^{-1/2}
	\end{aligned}
	\]

By Theorem \ref{mainthm-all indices are 0 or 1} and the facts $\f{1}{\gamma_\nu}\le \f{1}{\gamma_{\nu+e_1}}$ and $\f{1}{1-\gamma_{\nu+e_1}}\le \f{1}{1-\gamma_{\nu}} $ (since $\sigma(e_1)=e_1$) and the inductive hypothesis, the Riesz transforms $\delta_{\nu+e_1}^{k-e_1}(\mathcal L_{\nu+e_1}+2)^{-(|k|-1)/2}$ and $ \delta_{\nu_1}\mathcal L_\nu^{-1/2}$ are bounded on $L^p(\Rn_+)$ for $\f{1}{1-\gamma_\nu}<p<\f{1}{\gamma_\nu}$. Hence, it suffices to prove that the operator $\displaystyle \f{1}{x_1}\delta_{\nu+e_1}^{k-2e_1}(\mathcal L_{\nu+e_1}+2)^{-(|k|-1)/2}$ for on $L^p(\Rn_+)$ for $\f{1}{1-\gamma_\nu}<p<\f{1}{\gamma_\nu}$.

Indeed, using the formula
\[
(\mathcal L_{\nu+e_1}+2)^{-(|k|-1)/2}=c\int_0^\vc t^{(|k|-1)/2}e^{-2t} e^{-t\mathcal L_{\nu+e_1}} \f{dt}{t},
\]
we have
\[
\begin{aligned}
	\Big|\f{1}{x_1}\delta_{\nu_1+e_1}^{k-2e_1}&(\mathcal L_{\nu+e_1}+2)^{-(|k|-1)/2}f(x)\Big|\\
	&\lesi \int_{\mathbb R_+}\int_0^\vc \f{\sqrt t}{x_1} t^{(k_1-2)/2}|\delta_{\nu_1+1}^{k_1-2} p_t^{\nu_1+1}(x_1,y_1)|\prod_{i\ne 1}\sup_{t_i>0}\int_{\mathbb R_+} |t^{k_i/2}\delta_{\nu_i}^{k_i}p_{t_i}^{\nu_i}(x_i,y_i)||f(y)|d(y_i)\f{dt}{t}d{y_j}.
\end{aligned}
\]
From Theorem \ref{thm-delta k pnu} and Propositions \ref{prop- delta k pt nu -1/2}, \ref{prop1-boundedness} and \ref{prop2 - boundedness}, we have the following operators
\[
f\mapsto \sup_{t_i>0}\int_{\mathbb R_+} |t^{k_i/2}\delta_{\nu_i}^{k_i}p_{t_i}^{\nu_i}(x_i,y_i)||f(y_i)|d(y_i), i\ne 1
\]
and
\[
f\mapsto \int_{\mathbb R_+}\int_0^\vc \f{\sqrt t}{x_1} t^{(k_1-2)/2}|\delta_{\nu_1+1}^{k_1-2} p_t^{\nu_1+1}(x_1,y_1)|f(y_1)|\f{dt}{t}dy_1
\]
are bounded on $L^p(\mathbb R_+)$ for $\f{1}{1-\gamma_\nu}<p<\f{1}{\gamma_\nu}$.

It follows that the operator $\displaystyle \f{1}{x_1}\delta_{\nu+e_1}^{k-2e_1}(\mathcal L_{\nu+e_1}+2)^{-(|k|-1)/2}$ is bounded on $L^p(\Rn_+)$ for $\f{1}{1-\gamma_\nu}<p<\f{1}{\gamma_\nu}$. 

Hence, we have proved that the Riesz transform  $\delta_\nu^k (\mathcal L_\nu+a)^{-|k|/2}$ is bounded on $L^p(\Rn_+)$ for all $\f{1}{1-\gamma_\nu}<p<\f{1}{\gamma_\nu}$.

\medskip

We now turn to prove the weighted estimates. We need only to prove for the case $a=0$. Fix $\f{1}{1-\gamma_\nu}<p<\f{1}{\gamma_\nu}$ and $w\in A_{p(1-\gamma_\nu)}\cap RH_{(\f{1}{p\gamma_\nu})'}$. Then by Lemma \ref{weightedlemma1}, we can find $\f{1}{1-\gamma_{\nu}}<p_0<p<q_0<\f{1}{\gamma_\nu}$ such that $w\in A_{p_0(1-\gamma_{\nu})}\cap RH_{(\f{1}{q_0\gamma_\nu})'}$. For $m\in \mathbb N$ such that $2m-1>n$, by using the arguments used in the proofs of \eqref{eq1-proof of mainthm} and \eqref{eq2-proof of mainthm}, we also obtain that   
\begin{equation*}
	\begin{aligned}
		\Big(\fint_{S_j(B)}|\delta_\nu^k\mathcal L^{-|k|/2}_\nu(I-e^{-r_B^2\mathcal L_\nu})^mf|^{q_0} dx \Big)^{1/q_0}\lesi 2^{-j(2m-1)}\Big(\fint_B |f|^{q_0}dx\Big)^{1/q_0}
	\end{aligned}
\end{equation*}
and
\begin{equation*}
	\begin{aligned}
		\Big(\fint_{S_j(B)}|[I-(I-e^{-r_B^2\mathcal L_\nu})^m]f|^{q_0}dx\Big)^{1/q_0}\leq
		2^{-j(n+1)}\Big(\fint_B |f|^{p_0}dx\Big)^{1/p_0} 
	\end{aligned}
\end{equation*}
for all $j\ge 2$ and $f\in L^p(\mathbb R^n_+)$ supported in $B$.
	
Hence, the conclusion follows directly from these two estimates and Theorem \ref{BZ-thm}.

This completes our proof.
\end{proof}

In order to prove the main theorem, we also need the following result.
\begin{thm}\label{mainthm-all odds are 1}
	Let $\nu\in (-1,\vc)^n$. For  $k \in \mathbb N^n$ such that all odd entries of $k$ are equal to $1$, the Riesz transform $\delta_\nu^k (\mathcal L_\nu+a)^{-|k|/2}$ is bounded on $L^p_w(\Rn_+)$ for all $\f{1}{1-\gamma_\nu}<p<\f{1}{\gamma_{\nu+\sigma(k)}}$ and $w\in A_{p(1-\gamma_\nu)}\cap RH_{(\f{1}{p\gamma_{\nu+\sigma(k)}})'}$.
\end{thm}
\begin{proof}
	Assume that the first $j$ entries of $k$ are equal to $1$ and other entries are even. By \eqref{eq- delta and eigenvector}, 
	\[
	\delta_\nu^k \mathcal L_\nu^{-|k|/2} = \prod_{i=j+1}^n\delta_{\nu_i}^{k_i} \mathcal L_{\nu+\sigma(k)}^{-(n-j)/1}\circ \prod_{i=1}^j\delta_{\nu_i} \mathcal L_\nu^{-j/2}.
	\]
	By Theorems \ref{mainthm-all indices are 0 or 1} and \ref{mainthm-all even indices - weighted estimate}, the Riesz transforms $\prod_{i=j+1}^n\delta_{\nu_i}^{k_i} \mathcal L_{\nu+\sigma(k)}^{-(n-j)/1}$ and $\prod_{i=1}^j\delta_{\nu_i} \mathcal L_\nu^{-j/2}$ are bounded on $L^p_w(\Rn_+)$ for all $\f{1}{1-\gamma_\nu}<p<\f{1}{\gamma_{\nu+\sigma(k)}}$ and $w\in A_{p(1-\gamma_\nu)}\cap RH_{(\f{1}{p\gamma_{\nu+\sigma(k)}})'}$.
	
	This completes the proof. 
\end{proof}
We are ready to give the proof of Theorem \ref{mainthm-genral case}.

\begin{proof}[Proof of Theorem \ref{mainthm-genral case}:]
	Due to Theorems \ref{mainthm-all even indices - weighted estimate} and \ref{mainthm-all odds are 1}, we need only to prove for the case $k_{\rm odd}\ne 0$ and at least one odd index greater than or equal to  $3$. Without loss of generality, we might assume that  the first $j$ entries $k_1,\ldots,k_j$ are odd and greater than $1$ (and hence, are greater than or equal to $3$), the next $\ell$ entries $k_{j+1}=\ldots=k_{j+\ell}=1$ for some $j,\ell\in \{1,\ldots,n\}$ with $j+\ell\le n$, and the remaining entries $k_{j+\ell+1},\ldots, k_n$ are even. In this case, $\sigma(k)=(1,\ldots, 1, 0,\ldots, 0)$, i.e., the first $(j+\ell)$ entries are equal to $1$ and other are $0$. Then, using \eqref{eq- delta and eigenvector} we have
	\[
	\delta_\nu^k \mathcal L_\nu^{-|k|/2} = \delta_\nu^{k-\sigma(k)}(\mathcal L_{\nu+\sigma(k)}+2(j+\ell))^{-|k-\sigma(k)|/2}\circ  \delta^{\sigma(k)}_{\nu}\mathcal L_\nu^{-|\sigma(k)|/2}.
	\]
	The weighted boundedness of $\delta^{\sigma(k)}_{\nu}\mathcal L_\nu^{-|\sigma(k)|/2}$ follows by Theorem \ref{mainthm-all indices are 0 or 1}.
	
	It suffices to prove the weighted estimates of $\delta_\nu^{k-\sigma(k)}(\mathcal L_{\nu+\sigma(k)}+2(j+\ell))^{-|k-\sigma(k)|/2}$. To do this, by \eqref{eq- del nu and del nu + 1}, for $i=1, \ldots, j$ we have
	\[
	\delta_{\nu_i}^{k_i-1} =\delta_{\nu_i+1}^{k_i-1}+\f{1}{x_i}\delta_{\nu_i+1}^{k_i-2}.
	\]
	Consequently,
	\[
	\begin{aligned}
		\delta_\nu^{k-\sigma(k)} &=\prod_{i=1}^j \Big(\delta_{\nu_i+1}^{k_i-1}+\f{1}{x_i}\delta_{\nu_i+1}^{k_i-2}\Big)\prod_{i=j+\ell+1}^n \delta_{\nu_i}^{k_i}\\
		&=\prod_{i=1}^j  \delta_{\nu_i+1}^{k_i-1} \prod_{i=j+\ell+1}^n \delta_{\nu_i}^{k_i} + \sum_{\sigma \in \{0,1\}^j\backslash \{\textbf{0}\}}  \prod_{i=1}^j \frac{1}{x_i^{\sigma_i}}  \delta_{\nu_i+1}^{k_i - 1 - \sigma_i}\prod_{i=j+\ell+1}^n \delta_{\nu_i}^{k_i}.
	\end{aligned}
	\]
	
	Hence,
	\[
	\begin{aligned}
		\delta_\nu^{k-\sigma(k)}&(\mathcal L_{\nu+\sigma(k)}+2(j+\ell))^{-|k-\sigma(k)|/2}\\
		&=\prod_{i=1}^j  \delta_{\nu_i+1}^{k_i-1} \prod_{i=j+\ell+1}^n \delta_{\nu_i}^{k_i} (\mathcal L_{\nu+\sigma(k)}+2(j+\ell))^{-|k-\sigma(k)|/2}\\
		& \ \ \ + \sum_{\sigma \in \{0,1\}^j\backslash \{\textbf{0}\}}  \prod_{i=1}^j \frac{1}{x_i^{\sigma_i}}  \delta_{\nu_i+1}^{k_i - 1 - \sigma_i}\prod_{i=j+\ell+1}^n \delta_{\nu_i}^{k_i}(\mathcal L_{\nu+\sigma(k)}+2(j+\ell))^{-|k-\sigma(k)|/2}.
	\end{aligned}
	\]
	By Theorem \ref{mainthm-all even indices - weighted estimate} and the fact $\f{1}{1-\gamma_{\nu+\sigma(k)}}<\f{1}{1-\gamma_\nu}$ the first operator in the above identity  is bounded on $L^p_w(\Rn_+)$ for all  $\f{1}{1-\gamma_\nu}<p<\f{1}{\gamma_{\nu+\sigma(k)}}$ and $w\in A_{p(1-\gamma_\nu)}\cap RH_{(\f{1}{p\gamma_{\nu+\sigma(k)}})'}$. For the second operator we have, let $\sigma \in \{0,1\}^j\backslash \{\textbf{0}\}$. Without loss of generality, we might assume $\sigma_1=1$. Then,  we have
	\[
\begin{aligned}
	\Big|\prod_{i=1}^j \frac{1}{x_i^{\sigma_i}}  &\delta_{\nu_i+1}^{k_i - 1 - \sigma_i}\prod_{i=j+\ell+1}^n \delta_{\nu_i}^{k_i}(\mathcal L_{\nu+\sigma(k)}+2(j+\ell))^{-|k-\sigma(k)|/2}f(x)\Big|\\
	&\lesi \int_{\mathbb R_+}\int_0^\vc  \Big|\f{\sqrt t}{x_1} t^{(k_1-2)/2}\delta^{k_1-2}_{\nu_1+1}p_t^{\nu_1+1}(x_1,y_1)\Big|\\
	&\prod_{i\in\{2,\ldots, j\}\atop  \sigma_i\ne 0 }\sup_{t_i>0}\int_{\mathbb R_+}\Big|\f{\sqrt t_i}{x_i} t_i^{(k_i-2)/2}\delta^{k_i-2}_{\nu_i+1}p_t^{\nu_i+1}(x_i,y_i)\Big|\prod_{i\in\{2,\ldots, j\}\atop  \sigma_i=0}\sup_{t_i>0}\int_{\mathbb R_+} |t_i^{(k_i-1)/2}\delta_{\nu_i+1}^{k_i-1}p_{t_i}^{\nu_i+1}(x_i,y_i)|\\
	& \ \ \ \ \times \prod_{i=j+\ell+1}^n\sup_{t_i>0}\int_{\mathbb R_+} |t^{k_i/2}\delta_{\nu_i}^{k_i}p_{t_i}^{\nu_i}(x_i,y_i)||f(y)|d(y_i)\f{dt}{t}d{y_j}.
\end{aligned}
	\]
	Note that since $k_i-2$ is odd for $i=1,\ldots, j$, due to Proposition \ref{prop-2-stronger estimate on delta k },  for $i\in\{2,\ldots, j\}$ with $\sigma_i=0$, 
	\[
	\begin{aligned}
		\f{\sqrt t_i}{x_i} t_i^{(k_i-2)/2}|\delta^{k_i-2}_{\nu_i+1}p_t^{\nu_i+1}(x_i,y_i)|&\lesi \f{1}{\sqrt t} \exp\Big(-\f{|x-y|^2}{ct}\Big) \Big(1+\f{\sqrt t}{x}\Big)^{-(\nu_i+3/2)}\\
		&\lesi \f{1}{\sqrt t} \exp\Big(-\f{|x-y|^2}{ct}\Big),
	\end{aligned}
	\]
	which implies that 
	\[
	 \sup_{t_i>0}\int_{\mathbb R_+}\Big|\f{\sqrt t_i}{x_i} t_i^{(k_i-2)/2}\delta^{k_i-2}_{\nu_i+1}p_t^{\nu_i+1}(x_i,y_i)\Big| |f(y_i)| dy_i \lesi \mathcal M f(x_i).
	\]
	Hence,
	\[
	f\mapsto  \sup_{t_i>0}\int_{\mathbb R_+}\Big|\f{\sqrt t_i}{x_i} t_i^{(k_i-2)/2}\delta^{k_i-2}_{\nu_i+1}p_t^{\nu_i+1}(x_i,y_i)\Big| |f(y_i)| dy_i 
	\]
		is bounded on $L^p(\mathbb R_+)$ for all $1<p<\vc$.
		
	For the same reason, by Proposition \ref{prop- delta k pt nu -1/2}, for $i\in \{1,\ldots, j\}$ with $\sigma_i=0$ the operator 
	\[
	f\mapsto \sup_{t_i>0}\int_{\mathbb R_+}|t_i^{(k_i-1)/2}\delta_{\nu_i+1}^{k_i-1}p_{t_i}^{\nu_i+1}(x_i,y_i)| |f(y_i)| dy_i
	\]
	is dominated by (a multiple of) the Hardy-Littlewood maximal function $\mathcal M f(x_i)$ and hence, is bounded on $L^p(\mathbb R_+)$ for all $1<p<\vc$.
	
	Finally, by Theorem  \ref{thm-delta k pnu} and Proposition \ref{prop1-boundedness},  the operator 
	\[
	f \mapsto \sup_{t_i>0}\int_{\mathbb R_+} |t^{k_i/2}\delta_{\nu_i}^{k_i}p_{t_i}^{\nu_i}(x_i,y_i)||f(y_i)|d(y_i)
	\]
	is bounded on $L^p(\mathbb R_+)$ for all $\f{1}{1-\gamma_{\nu_i}}<p<\f{1}{\gamma_{\nu_i}}$ for $i=j+\ell+1,\ldots, n$; and hence is bound $L^p(\mathbb R_+)$ for all $\f{1}{1-\gamma_\nu}<p<\f{1}{\gamma_{\nu+\sigma(k)}}$ (since $\gamma_{\nu+\sigma(k)}\ge \gamma_{\nu_i}$ for $i=j+\ell+1,\ldots, n$). 
	
	It follows that the operators 
	\[
	\prod_{i=1}^j \frac{1}{x_i^{\sigma_i}}  \delta_{\nu_i+1}^{k_i - 1 - \sigma_i}\prod_{i=j+\ell+1}^n \delta_{\nu_i}^{k_i}(\mathcal L_{\nu+\sigma(k)}+2(j+\ell))^{-|k-\sigma(k)|/2}
	\]
	is bounded on $L^p(\mathbb R_+)$ with $\f{1}{1-\gamma_\nu}<p<\f{1}{\gamma_{\nu+\sigma(k)}}$ for all $\sigma \in \{0,1\}^j\backslash \{\textbf{0}\}$.
	
	To complete the proof it suffices to prove that for each  $\sigma \in \{0,1\}^j\backslash \{\textbf{0}\}$, the  operator
	\[
	   \prod_{i=1}^j \frac{1}{x_i^{\sigma_i}}  \delta_{\nu_i+1}^{k_i - 1 - \sigma_i}\prod_{i=j+\ell+1}^n \delta_{\nu_i}^{k_i}(\mathcal L_{\nu+\sigma(k)}+2(j+\ell))^{-|k-\sigma(k)|/2}
	\]
		is bounded on $L^p(\mathbb R_+)$ with $\f{1}{1-\gamma_\nu}<p<\f{1}{\gamma_{\nu+\sigma(k)}}$  and $w\in A_{p(1-\gamma_\nu)}\cap RH_{(\f{1}{p\gamma_{\nu+\sigma(k)}})'}$.
		
		To do this,  fix $\f{1}{1-\gamma_\nu}<p<\f{1}{\gamma_{\nu+\sigma(k)}}$ and $w\in A_{p(1-\gamma_\nu)}\cap RH_{(\f{1}{p\gamma_{\nu+\sigma(k)}})'}$. By Lemma \ref{weightedlemma1}, we can find $\f{1}{1-\gamma_{\nu}}<p_0<p<q_0<\f{1}{\gamma_{\nu+\sigma(k)}}$ such that $w\in A_{{p_0(1-\gamma_{\nu})}}\cap RH_{(\f{1}{q_0\gamma_{\nu+\sigma(k)}})'}$.  For each $\sigma \in \{0,1\}^j\backslash \{\textbf{0}\}$, we set 
	\begin{equation}\label{eq- Tsigma}
	\begin{aligned}
		T_\sigma:&=\prod_{i=1}^j \frac{1}{x_i^{\sigma_i}}  \delta_{\nu_i+1}^{k_i - 1 - \sigma_i}\prod_{i=j+\ell+1}^n \delta_{\nu_i}^{k_i}(\mathcal L_{\nu+\sigma(k)}+2(j+\ell))^{-|k-\sigma(k)|/2}\\
		&=\prod_{i\in \{1,\ldots,j\}\atop \sigma_i=1} \frac{1}{x_i}  \delta_{\nu_i+1}^{k_i - 2}\prod_{i\in \{1,\ldots,j\}\atop \sigma_i=0} \delta_{\nu_i+1}^{k_i - 1}\prod_{i=j+\ell+1}^n \delta_{\nu_i}^{k_i}(\mathcal L_{\nu+\sigma(k)}+2(j+\ell))^{-|k-\sigma(k)|/2}\\
		&=\prod_{i\in \{1,\ldots,j\}\atop \sigma_i=1} \frac{1}{x_i}  \delta_{\nu_i+1}^{k_i - 2}\prod_{i\in \{1,\ldots,j\}\atop \sigma_i=0} \delta_{\nu_i+1}^{k_i - 1}\prod_{i=j+\ell+1}^n \delta_{\nu_i}^{k_i}(\mathcal L_{\nu+\sigma(k)}+2(j+\ell))^{-|k-\sigma(k)|/2}.
	\end{aligned}
	\end{equation}
	
	Due to Theorem \ref{BZ-thm}, we need only to verify the following two estimates
	\begin{equation}\label{eq1-proof of last thm}
		\begin{aligned}
			\Big(\fint_{S_j(B)}|T_{\sigma}(I-e^{-r_B^2\mathcal L_{\nu+\sigma(k)}})^mf|^{q_0} dx \Big)^{1/q_0}\lesi 2^{-j(2m-1)}\Big(\fint_B |f|^{q_0}dx\Big)^{1/q_0}
		\end{aligned}
	\end{equation}
	and
	\begin{equation}\label{eq2-proof of last thm}
		\begin{aligned}
			\Big(\fint_{S_j(B)}|[I-(I-e^{-r_B^2\mathcal L_{\nu+\sigma(k)} })^m]f|^{q_0}dx\Big)^{1/q_0}\leq
			2^{-j(n+1)}\Big(\fint_B |f|^{p_0}dx\Big)^{1/p_0} 
		\end{aligned}
	\end{equation}
	for $2m-1>n$ and  all $j\ge 2$ and $f\in L^p(\mathbb R^n_+)$ supported in $B$.
	
	Using the identity
	\[
	I-(I-e^{-r_B^2\mathcal L_{\nu+\sigma(k)} })^m =\sum_{i=1}^m c_i e^{-ir_B^2\mathcal L_{\nu+\sigma(k)} },
	\]
	where $c_i$ are constants, \eqref{eq-partial pt} in Proposition \ref{prop- delta  pt} and Lemma \ref{lem Lpq}, we will obtain \eqref{eq2-proof of last thm}.
	
	For the first inequality, from \eqref{eq- Tsigma} and the following identities
	\[
	(\mathcal L_{\nu+\sigma(k)}+2(j+\ell))^{-|k-\sigma(k)|/2}=c\int_0^\vc t^{|k-\sigma(k)|/2}e^{-2(j+\ell)t} e^{-t\mathcal L_{\nu+\sigma(k)}}
	\]
	and
		\[
	(I-e^{-r_B^2\mathcal L_{\nu+\sigma(k)}})^m=\int_{[0,r_B^2]^m}\mathcal L_{\nu+\sigma(k)}^m e^{-(s_1+\ldots+s_m)\mathcal L_{\nu+\sigma(k)}}d\vec{s},
	\]
	where $d\vec{s}=ds_1\ldots ds_m$, we have
	\[
	\begin{aligned}
		T_{\sigma}&(I-e^{-r_B^2\mathcal L_{\nu+\sigma(k)}})^mf\\
		& = c\int_0^\vc\int_{[0,r_B^2]^m}t^{|k-\sigma(k)|/2}e^{-2(j+\ell)t}\\
		& \ \ \times \prod_{i\in \{1,\ldots,j\}\atop \sigma_i=1} \frac{1}{x_i}  \delta_{\nu_i+1}^{k_i - 2}\prod_{i\in \{1,\ldots,j\}\atop \sigma_i=0} \delta_{\nu_i+1}^{k_i - 1}\prod_{i=j+\ell+1}^n \delta_{\nu_i}^{k_i} \mathcal L_{\nu+\sigma(k)}^m e^{-(t+s_1+\ldots+s_m)\mathcal L_{\nu+\sigma(k)}}f d\vec{s} \f{dt}{t}.
	\end{aligned}
	\]
	Note that in this case all $k_i-2$ are even for $i\in \{1,\ldots, j\}$ with $\sigma_i=1$, and both $k_i, i=j+\ell+1,\ldots, n$ and  $k_i-1, i\in \{1,\ldots, j\}$ with $\sigma_i=0$ are even. Hence, we can arguing similarly to the proof of \eqref{eq2-proof of mainthm dual} by using  Theorem \ref{thm3- kernel est  n ge 2} and Lemma \ref{lem Lpq} to obtain \eqref{eq2-proof of last thm}.
	
	Hence, this completes our proof.
		
\end{proof}

{\bf Acknowledgement.} The author was supported by the research grant ARC DP140100649 from the Australian Research Council.

\end{document}